\documentclass[12pt]{amsart}
\usepackage{amssymb, amsmath, amsfonts, amsthm, graphics}
\usepackage[hmargin=1 in, vmargin = 1 in]{geometry}
\usepackage[T1]{fontenc}
\usepackage{libertinus}

\usepackage{epstopdf}
\usepackage{graphicx}
\usepackage{tikz-cd}
\usepackage{hyperref}
\usepackage{enumitem}
\usepackage[all]{xy}
\usepackage{stmaryrd}
\usepackage{booktabs}

\usepackage{mathrsfs}
\usepackage[most]{tcolorbox}
\usepackage[most]{tcolorbox}
\usepackage{xcolor}

\tcbset{
  myhighlight/.style={
    colback=yellow!20,
    colframe=yellow!50!black,
    boxrule=0.4pt,
    arc=1mm,
    left=1mm,right=1mm,top=1mm,bottom=1mm
  }
}
\usepackage{mathabx,epsfig}

\usepackage{amsthm}

\numberwithin{equation}{section}
\newtheorem{theorem}{Theorem}[section]        
\newtheorem{proposition}[theorem]{Proposition}

\newtheorem{lemma}[theorem]{Lemma}
\newtheorem{corollary}[theorem]{Corollary}
\newtheorem{example}[theorem]{Example}

\theoremstyle{definition}

\newtheorem{definition}[theorem]{Definition}
\newtheorem{remark}[theorem]{Remark}

\newcommand{\restr}[2]{\left.\kern-\nulldelimiterspace #1 \vphantom{\big|}\right|_{#2}}

\newcommand{\jet}[1]{ [\hspace{-0,5mm}[ {#1} ]\hspace{-0,5mm}] }
\newcommand{\rjet}[1]{ (\hspace{-0,7mm}( {#1} )\hspace{-0,7mm}) }

\newcommand{\Bun}{\mathrm{Bun}}

\DeclareMathOperator{\Aut}{Aut}
\DeclareMathOperator{\Pic}{Pic}
\DeclareMathOperator{\Hom}{Hom}
\DeclareMathOperator{\Ext}{Ext}
\DeclareMathOperator{\End}{End}

\DeclareMathOperator{\Ker}{Ker}
\DeclareMathOperator{\Coker}{Coker}
\DeclareMathOperator{\Spec}{Spec}

\newcommand{\im}{\operatorname{im}}
\renewcommand{\Im}{\mathop{\mathrm{Im}}}

\DeclareMathOperator{\rk}{rk}
\DeclareMathOperator{\gr}{gr}
\DeclareMathOperator{\supp}{supp}

\newcommand{\GL}{\mathrm{GL}}
\newcommand{\PGL}{\mathrm{PGL}}
\newcommand{\mat}[4]{\begin{pmatrix} #1 & #2 \\ #3 & #4 \end{pmatrix}}
\newcommand{\smat}[4]{\left(\begin{smallmatrix} #1 & #2 \\ #3 & #4 \end{smallmatrix}\right)}

\usepackage[most]{tcolorbox}
\usepackage{xcolor}

\tcbset{
  myhighlight/.style={
    colback=yellow!20,
    colframe=yellow!50!black,
    boxrule=0.4pt,
    arc=1mm,
    left=1mm,right=1mm,top=1mm,bottom=1mm
  }
}

\renewcommand{\AA}{\mathbb{A}}

\newcommand{\CC}{\mathbb{C}}

\newcommand{\FF}{\mathbb{F}}
\newcommand{\GG}{\mathbb{G}}

\newcommand{\PP}{\mathbb{P}}

\newcommand{\RR}{\mathbb{R}}

\newcommand{\cE}{\mathcal{E}}
\newcommand{\cF}{\mathcal{F}}

\newcommand{\cL}{\mathcal{L}}

\newcommand{\cO}{\mathcal{O}}
\newcommand{\cP}{\mathcal{P}}
\newcommand{\cQ}{\mathcal{Q}}

\newcommand{\cU}{\mathcal{U}}

\newcommand{\fg}{\mathfrak{g}}

\theoremstyle{definition}

\theoremstyle{definition}

\usepackage{amsmath,amssymb}

\makeatletter
\newcommand{\xRrightarrow}[2][]{\ext@arrow 0359\Rrightarrowfill@{#1}{#2}}
\newcommand{\Rrightarrowfill@}{\arrowfill@\equiv\equiv\Rrightarrow}
\newcommand{\xLleftarrow}[2][]{\ext@arrow 3095\Lleftarrowfill@{#1}{#2}}
\newcommand{\Lleftarrowfill@}{\arrowfill@\Lleftarrow\equiv\equiv}
\newcommand{\fullstopbelow}{\makebox[0pt][l]{\,.}} 

\makeatother

\begin{document}

\title{Eigenforms and graphs of Hecke operators with wild ramification}
\author{Rudrendra Kashyap}
\email{ruk26@pitt.edu}
\address{Rudrendra Kashyap, University of Pittsburgh, Pittsburgh,
PA, USA}

\author{Vladyslav Zveryk}
\email{zverik.vladislav@gmail.com}
\address{Vladyslav Zveryk, Yale University, New Haven,
CT, USA}

\keywords{Hecke operators, Hecke graphs, automorphic forms over function fields, ramification, moduli of bundles, Harder--Narasimhan stratification, Hecke eigenspaces}
\thanks{The first author was partially supported by NSF grant DMS-2402553. The second author was partially supported by NSF grant DMS-2501558.}

\begin{abstract}
Hecke operators on moduli of bundles over a global function field become substantially more complicated in the presence of ramification. We show that far enough in the Harder-Narasimhan cone of $\Bun_G$, this extra complexity has a simple structure, which allows to reduce most of the study to the unramified case. Using the theory of graphs of Hecke operators, we transform this statement into a combinatorial condition. Utilizing the combinatorial language, we obtain tight bounds, and for generic eigenvalues exact formulas, for the dimensions of Hecke eigenspaces with arbitrary ramification for \(\Bun_{\PGL_2}\). We compare these formulas to the known results in the theory of Eisenstein series. Moreover, our methods allow to construct eigenforms explicitly.

\end{abstract}
\maketitle
\tableofcontents
\section{Introduction}
Hecke operators for bundles over global function fields admit both an adelic and a geometric description. On the adelic side, isomorphism classes of principal \(G\)-bundles on a smooth projective curve \(X\) over a finite field are described by a double quotient, and imposing level structure along an effective divisor \(D\) amounts to replacing the standard maximal compact subgroup by a smaller compact open subgroup. Hecke operators are then defined by convolution with elements of the corresponding Hecke algebra. On the geometric side, the same operators are realized by Hecke correspondences on the moduli stack \(\Bun_{G,D}\) of \(G\)-bundles with \(D\)-level structure. More precisely, for a closed point \(x\in |X|\) and a dominant coweight \(\mu\in X_*(T)^+\), one considers the correspondence \(\mathrm{Hecke}^\mu_{G,D,x}\) classifying pairs of \(G\)-bundles with \(D\)-level structure related by a Hecke modification of type \(\mu\) at \(x\), with compatibility with the prescribed level structure away from \(x\) and with the local \(K(D)_x\Delta_x^\mu K(D)_x\)-condition at \(x\) when \(x\in \operatorname{supp}(D)\). Passing to \(k\)-points yields the Hecke graphs studied in this paper: the vertices are isomorphism classes of bundles with level structure, and the edges record Hecke modifications of the chosen type.

This graph-theoretic packaging of Hecke operators was developed systematically by Lorscheid \cite{Lor13} and later by Alvarenga and others in several directions, including higher rank, elliptic curves, and explicit spectral questions for low-genus curves \cite{Alv19,Alv20,AB24}. It is particularly useful when one seeks explicit information about automorphic forms over function fields, since it turns Hecke operators into concrete adjacency operators on combinatorial objects. At the same time, the underlying geometry remains that of Hecke correspondences on moduli stacks of bundles, so the graph-theoretic viewpoint may be regarded as a concrete shadow of the geometry of \(\Bun_G\) and its Hecke modifications.

\subsection{Main results}

Let \(X\) be a smooth geometrically connected projective curve over a finite field \(k=\mathbb F_q\), let \(F=\mathbb F_q(X)\) be its function field, and let \(G\) be a connected split reductive group over \(k\). For an effective divisor \(D\) on \(X\), we write \(\Bun_{G,D}\) for the moduli stack of \(G\)-bundles with \(D\)-level structure. Fix a closed point \(x\in |X|\), and consider the Hecke graph attached to a Hecke modification at \(x\). Its vertices are isomorphism classes of \(k\)-points of \(\Bun_{G,D}\), and its edges are given by Hecke modifications at \(x\). In other words, the Hecke graph is the graph whose adjacency matrix is precisely the matrix of the Hecke operator acting on functions on \(\Bun_{G,D}\). The basic question is then the following: how does the added ramification along \(D\) affect the Hecke graph, and does this effect stabilize in the cuspidal region?

Our first group of results answers this question. Fix a decomposition $D=D_1+D_2$ of $D$ into a sum of effective divisors such that
\[
\operatorname{supp}(D_1)\cap\operatorname{supp}(D_2)=\varnothing,
\qquad
x\notin\operatorname{supp}(D_1).
\]

We give an explicit description of the fibers of the forgetful map
\[
\Bun_{G,D}\longrightarrow \Bun_{G,D_2}
\]
over bundles lying deep in the Harder-Narasimhan cusp (see Definition \ref{def:Dcusp}). The main statement is the following.
\begin{theorem}\
\begin{enumerate}[label=(\roman*)]
    \item If \(\operatorname{supp}(D_2)\neq\{x\}\), then the forgetful map on the full
subgraphs on the deep cusp vertices
\begin{equation}\label{eq:projection_restricted_to_cusps}
    p_{D,D_2}:\Bun_{G,D}^{\mathrm{cusp}}
\longrightarrow
\Bun_{G,D_2}^{\mathrm{cusp}}
\end{equation}
is a disjoint covering of constant degree
\[
N=
\frac{|T(\mathcal O_{D_1})|}
{|T(k)|^{\delta_{D_2,0}}}
\cdot
|(G/B)(\mathcal O_{D_1})|,
\]
where $\delta$ is the Kronecker delta function.
\item If \(\operatorname{supp}(D_2)=\{x\}\), then the map (\ref{eq:projection_restricted_to_cusps}) is an $N:1$ topological covering of graphs with monodromy given by $T(k)$.
\end{enumerate}

\end{theorem}

This is Theorems~\ref{t:graph_for_ramification_at_cusp} and
\ref{t:graph_for_ramification_at_cusp_over_x} for $\GL_n$ and Theorems \ref{thm:deep_cusp_splitting_reductive} and
\ref{thm:deep_cusp_torus_monodromy_reductive} for general $G$.

The second theme of the paper is spectral. For a graph $\Gamma$ and $\lambda\in\CC$, let $\dim_\lambda\Gamma$ be the dimension of the $\lambda$-eigenvalues of the adjacency matrix of $\Gamma$ acting on $\CC\Gamma$. Our main result in this section is the following.
\begin{theorem}
    Choose $\lambda\in \CC$, and let a graph $\Gamma$ admit a decomposition $\Gamma=\Gamma'\sqcup\Gamma_1\sqcup\Gamma_2\sqcup\ldots$ such that the adjacency linear maps $A_i:\CC\Gamma_i\to\CC\Gamma_{i-1}$ induced by the arrows from vertices of $\Gamma_{i-1}$ to $\Gamma_i$ are surjective. Then 
    
    $$
    \sup_{i\ge 1}|\Gamma_i|\le\dim_\lambda\Gamma\le\dim_\lambda\Gamma'+\sup_{i\ge 1}|\Gamma_i|.
    $$
    
    In particular, if the maps $A_i$ are isomorphisms, then
    $$
    |\Gamma_1|\le\dim_\lambda\Gamma\le\dim_\lambda\Gamma'+|\Gamma_1|.
    $$
\end{theorem}
We show that Hecke graphs for $\PGL_2$ admit a similar type decomposition, with $\Gamma'$ being the so-called ``nucleus'' \cite{Lor13,Alv19,AB24}, and $\Gamma_i$ lying in the deep cusp of $\Bun_{\PGL_2,D}$.
This allows us to determine dimensions of spaces of $\PGL_2$-eigenforms. The results are the following. Set \(\Gamma:=\Bun_{G,D}\) to be the graph of a Hecke operator at $x$.
\begin{theorem}
Let $r:=\deg x$. There exists a finite subgraph $\Gamma'$ of $\Gamma$ depending on $D$ such that
    \begin{enumerate}[label=(\roman*)]
    \item If \(D=0\), then
\[
r|\Pic^0(X)(\mathbb F_q)|
\leq
\dim_\lambda\Gamma
\leq
\dim_\lambda\Gamma'
+
r|\Pic^0(X)(\mathbb F_q)|.
\]
    \item If $D\ne 0$ and \(x\notin\operatorname{supp}(D)\), then 
\[
r|\Pic^0(X)(\mathbb F_q)|
\frac{q^{2(\deg D-\deg D_{\mathrm{red}})}}{q-1}
\prod_{y\in\operatorname{supp}D}(q^{2\deg y}-1)
\leq
\dim_\lambda\Gamma
\]
and
\[
\dim_\lambda\Gamma
\leq
\dim_\lambda\Gamma'
+
r|\Pic^0(X)(\mathbb F_q)|
\frac{q^{2(\deg D-\deg D_{\mathrm{red}})}}{q-1}
\prod_{y\in\operatorname{supp}D}(q^{2\deg y}-1),
\]
where $D_{\mathrm{red}}:=\sum_{y\in\supp D}1\cdot[y]$.
\item If \(D=D'+d_x[x]\) with
\(d_x\geq1\) and $x\notin\supp D'$, then
\[
r|\Pic^0(X)(\mathbb F_q)|
\frac{(q^r-1)q^{2(\deg D'-\deg D'_{\mathrm{red}})+r(d_x-1)}}{q-1}
\prod_{y\in\operatorname{supp}D'}(q^{2\deg y}-1)
\leq
\dim_\lambda\Gamma
\]
and
$$
\dim_\lambda\Gamma
\leq
\dim_\lambda\Gamma'
+
r|\Pic^0(X)(\mathbb F_q)|
\frac{(q^r-1)q^{2(\deg D'-\deg D'_{\mathrm{red}})+r(d_x-1)}}{q-1}
\prod_{y\in\operatorname{supp}D'}(q^{2\deg y}-1).
$$
\end{enumerate}
In all cases, these inequalities are equalities for all but finitely many
\(\lambda\). 
\end{theorem}
These are Theorems~\ref{t:PGL2_ramified_not_at_x},
\ref{t:PGL2_ramified_at_x}, and
\ref{t:PGL2_ramified_at_x_and_D}. In addition, we extend these results to a general surjectivity theorem.
\begin{theorem}
    In all the above theorems, assume that $\lambda\in\CC$ is not an eigenvalue for the adjacency matrix of $\Gamma'$, and denote the Hecke operator in consideration by $\Phi$. Then the operator 
    $$
    \Phi-\lambda\cdot\mathrm{id}:\CC^\Gamma\to \CC^\Gamma
    $$
    is surjective. In particular, for any fixed $g\in\operatorname{Fun}(\Bun_{G,D}(k))$, the equation
    \begin{equation}
    (\Phi-\lambda)(f)=g
    \end{equation}
    has a finite-dimensional affine space of solutions whose dimension is given by the lower bound in the above theorem.
\end{theorem}
This is Theorem \ref{t:Phi-lambda_is_surjective}.

Our results place the asymptotic geometry of ramified Hecke graphs into a uniform framework built from Harder--Narasimhan theory, explicit control of torus monodromy, and propagation on infinite graphs. They also suggest a natural connection with geometric Langlands. Indeed, the graphs studied here are a combinatorial shadow of the same Hecke correspondences on \(\Bun_G\) that underlie Hecke functors and Hecke eigensheaves. Our focus in this paper is entirely on the explicit graph-theoretic and automorphic side, but the stabilization phenomenon we prove may be viewed as a concrete asymptotic manifestation of the geometry of Hecke correspondences with level structure.

The Hecke correspondences considered in this paper are the same geometric correspondences that appear in the geometric Langlands program. More precisely, Hecke correspondences define Hecke functors on categories of sheaves or \(D\)-modules on moduli stacks of bundles, and Hecke eigensheaves encode Langlands-dual local systems \cite{FGV02}. This connection remains meaningful in ramified settings. See, for example \cite{Dri87,Hei04,HNY13,Yun16,Bos22}. Thus, although the methods of the present paper are entirely automorphic and graph-theoretic, the objects studied here belong to the same general geometric framework as ramified Hecke eigensheaves. For this reason, the explicit ramified Hecke graphs studied in this paper may be viewed as concrete automorphic models attached to geometric Hecke correspondences with level structure.

We also note that we study a more general ramification in \cite{KZ2} building on results and methods of this paper. We give more examples there and obtain some structural differences from the ramification considered in this paper, e.g. the fact that the forgetful map $p_{D,D_2}$ does not always have same-sized fibers over the deep Harder-Narasimhan cusp.
\subsection{Overview}
The paper is organized as follows. In Section \ref{s:preliminaries_on_Hecke_graphs}, we fix notation and give preliminaries on Hecke graphs. 

In Section~\ref{section 2}, we develop the geometric and adelic description of ramified Hecke graphs. Subsections~\ref{Hecke Corres.GLn} and ~\ref{Hecke Corres.G} describe Hecke correspondences for \(\GL_n\) and general reductive groups with level structure respectively. 

Section~\ref{GLn case} proves the cusp structural results for \(\GL_n\), including splitting of the Harder--Narasimhan filtration, triangularity of automorphisms, fiber counts for the forgetful map, and the resulting decomposition of the ramified graph in the cusp. Subsection~\ref{s:General graphs} introduces the general propagation formalism for infinite graphs. Subsections~\ref{ss:unramified_PGL2} and~\ref{ss:ramified_PGL2} apply this formalism to spaces of eigenforms for \(\PGL_2\), first when the level is away from \(x\), and then when the level is ramified at \(x\). We also prove that for all but finitely many eigenvalues of the Hecke operator $\Phi$, the operator $\Phi-\lambda\cdot\mathrm{id}$ is surjective, which gives dimension formulas for its preimages, and in particular, generalized eigenspaces. Subsection~\ref{PGLn case} briefly discusses the higher-rank \(\PGL_n\) situation. 

Section~\ref{s:general_G} treats the general split reductive case: after fixing notation and standing assumptions, we define the relevant cusp loci, prove the reductive analogues of the \(\GL_n\) cusp results, classify cusp Hecke modifications of type \(\mu\), and establish the general cusp splitting theorem.

\section{Graphs of Hecke operators and notation}\label{s:preliminaries_on_Hecke_graphs}
For nice expositions, see \cite{Lor13,Alv19,AB24}. Fix the following notation, which we will use throughout the paper:
\begin{itemize}
    \item $X$ is a smooth projective geometrically connected curve $X$ over a finite field $k=\FF_q$.
    \item $|X|$ is the set of closed points in $X$. Choose $x\in |X|$.
    \item $F=k(X)$ is the function field, $\cO_y$ is the completion of the local ring of $X$ at $y$, $F_y$ is its fraction field. Let $\pi_y$ denote any uniformizer of $\cO_y$ and $k_x:=\cO_{x}/(\pi_x)$ the residue field of $x$.
    \item $\AA=\prod'_{y\in |X|}F_y$ is the adele ring of $X$.
    \item $D=\sum_y d_y[y]$ is an effective divisor over $X$. We often write $D=D'+d_x[x]$, where $D'=\sum_{y\ne x}d_y[y]$.
    \item $D=D_1+D_2$ for some effective divisors $D_1,D_2$ with $\supp D_1\cap\supp D_2=\varnothing$ and $x\notin\supp D_1$.
    \item $\Bun_{G,D}(\FF_q)$ is the set of isomorphism classes of principal $G$-bundles with ramification at $D$. A map $$p_{D,D_2}:\Bun_{G,D}(\FF_q)\to \Bun_{G,D_2}(\FF_q)$$ is the canonical projection forgetting the ramification at $D_1$. We set $p_D:=p_{D,0}$
    \item $\deg D:=\sum_y d_y\deg(y)$ and $D_{\mathrm{red}}:=\sum_{y\in\operatorname{supp} D}1\cdot [y]$ is the reduced divisor.
    \item $\cO_D:=\prod_{y\in |X|}\cO_{y}/\pi_{y}^{d_y}$ is the ring of functions on $D$.
    \item $G$ is a connected split reductive group over $k$, $B$ is a Borel subgroup and $T$ is a maximal torus in $B$.
    \item $K(D):=\prod_{y\in |X|}K(D)_y$ with $K(D)_y:=\ker(G(\cO_y)\to G(\cO_y/\pi_y^{d_y}))$.
\end{itemize}
For convenience, we also list here notation for some terms to be defined below:
\begin{itemize}
    \item $\mu$ is a dominant cocharacter of $T$ and $\Delta^\mu$ is the corresponding element of $K(D)\backslash G(\AA)/K(D)$.
    \item $\Phi^\mu_{D,x}$ is the Hecke operator corresponding to $\Delta^\mu$.
    \item $\Gamma^\mu_{D,x}$ is the graph of $\Phi^\mu_{D,x}$.
    \item $\dim_\lambda\Gamma^\mu_{D,x}$ denotes the dimension of the $\lambda$-eigenspace of $\Phi^\mu_{D,x}$.
\end{itemize}

Any smooth compactly-supported function 
$$
\Phi\in \operatorname{Fun}(K(D)\backslash G(\AA)/K(D),\CC)
$$
gives rise to an operator on the space of functions on $G(F)\backslash G(\AA)/K(D)$ via convolution:
$$
(\Phi f)(g)=\int_{G(\AA)}\Phi(h)f(gh)\, dh.
$$
Such an operator $\Phi$ is called a {\bfseries Hecke operator}. 

In this paper, we will be mostly interested in Hecke operators corresponding to characteristic functions on certain elements in $K(D)\backslash G(\AA)/K(D)$. To specify those, fix a uniformizer $\pi_x$ of $\cO_x$, equivalently, an isomorphism $k_x\rjet t\simeq F_x$ mapping $k_x\jet t$ to $\cO_x$. Define $\Delta^\mu\in G(\AA)$ via $\Delta_y^\mu=\mathrm{id}$ for $y\ne x$ and $\Delta_x^\mu$ as the image of $\mu$ via the series of inclusions
$$
X_{*}(T)=\Hom(\GG_m,T)\subset \Hom(\Spec k[t,t^{-1}],G)\subset \Hom(\Spec k_x\rjet t,G)\simeq G(F_x).
$$
\begin{remark}
{\rm
    If there is no ramification at $x$, then the double coset $K\Delta K$ does not depend on the choice of the uniformizer $\pi_x$. In general, if $d_x \ge 1$, the double coset $K(D)\Delta^\mu K(D)$ depends on the choice of uniformizer $\pi_x$ modulo $\pi_x^{d_x}$ (the ideal $\pi_x^{d_x}\cO_x$ is independent of any choices). Specifically, changing the uniformizer to $u\pi_x$ for $u \in \mathcal{O}_x^\times$ modifies the double coset by the left action of $\mu(u) \in T(\mathcal{O}_x)$. This is not an issue for us since we will work with one choice of the uniformizer throughout the whole paper.
}
\end{remark}

The Hecke operator corresponding to the characteristic function of $K(D)\Delta^\mu K(D)$ will be denoted by $\Phi^\mu_{D,x}$. We can also describe it as follows. Write a decomposition into left cosets
$$
K(D)\Delta^\mu K(D)=\bigsqcup_i\tau_i\Delta^\mu K(D)
$$
for some $\tau_i\in K(D)$. Note that since $\Delta^\mu_y$ is non-trivial only for $y=x$, we can (but don't have to) assume that $\tau_i$ satisfy the same property. Then (see the proof of \cite[Lemma 1.2]{Lor13})
$$
(\Phi^\mu_{D,x}f)(g)=\sum_{i}f(g\tau_i\Delta^\mu).
$$
Because of this, it makes sense to define the {\bfseries graph $\Gamma^\mu_{D,x}$ of $\Phi^\mu_{D,x}$} in the following way: the set of vertices is $G(F)\backslash G(\AA)/K(D)$, and $G(F)xK(D)$ is connected to $G(F)x\tau_i\Delta K(D)$ by a directed edge for all $i$. Note that with this definition, it is possible to have several edges between two vertices. We sometimes treat them as one edge and call their number the {\bfseries edge multiplicity}. With this description, we have
$$
(\Phi^\mu_{D,x}f)(g)=\sum_{g\to h}f(h),
$$
and therefore $\Phi^\mu_{D,x}$ becomes the adjacency matrix for $\Gamma^\mu_{D,x}$ if we treat the vertices of $\Gamma^\mu_{D,x}$ as the characteristic functions of elements of $G(F)\backslash G(\AA)/K(D)$.

An {\bfseries automorphic form} is a function on $G(F)\backslash G(\AA)/K(D)$ satisfying a certain moderate growth condition. For a precise meaning of this, we refer to \cite{AB24} since this condition will turn out to always be satisfied in our examples (see Remark \ref{r:bounded_growth}). We are interested in the dimension of the space of $\lambda$-eigenforms for $\Phi^\mu_{D,x}$, that is, automorphic forms $f$ satisfying $\Phi^\mu_{D,x}f=\lambda f$. As we just said, this will often coincide with the dimension of the space of all $\lambda$-eigenvectors for $\Phi^\mu_{D,x}$, which we will denote by $\dim_\lambda\Gamma^\mu_{D,x}$.

\section{Transition between adelic and geometric sides}\label{section 2}
In this section, we outline the geometric interpretation of Hecke operators and their graphs for an arbitrary smooth geometrically connected projective curve $X$ over $\FF_q$. For nice expositions in the $\GL_n$ case, see \cite{Lor13,Alv19,AKM25}.

First, recall the bijection
\begin{equation*}
    G(F)\backslash G(\AA)/G(\cO)\simeq \Bun_G(X)(\FF_q).
\end{equation*}
Note that the right-hand side is considered not as the groupoid of all $G$-bundles on $X$, but as the set of their equivalence classes. We will use this convention in the whole paper. The bijection goes as follows. Choose a principal bundle $\cP$ on $X$ and its trivializations $\phi_x:\cP|_{\cO_x}\simeq G_{\cO_x}$ for any $x\in X$, including the generic point $\eta$. This gives a composition of isomorphisms
$$
G_F=G_{\cO_x}|_\eta\xrightarrow{\phi_x^{-1}|_\eta}\cP|_{\cO_x}|_\eta=\cP_\eta\xrightarrow{\phi_\eta}G_F.
$$
Because we got a multiplication-invariant automorphism of $G_F$, it is given by multiplication by an element $g_x\in G(F)$. These form an element $(g_x)_{x\in |X|}\in G(\AA)$. This element depends on our choices of the trivializations $\phi_x$ which are unique up to the coordinate change by $G(\cO_x)$, and $\phi_\eta$ which is unique up to $G(F)$. Therefore, the element $(g_x)_{x\in |X|}\in G(\AA)$ is well-defined as an element of $G(F)\backslash G(\AA)/G(\cO)$.

We can generalize this bijection to the ramified case. Let $D$ be an effective divisor on $X$, and let $\cO_{D}$ be the ring of functions on $D$.

\begin{definition}
    The stack $\Bun_{G,D}(X)$ is given by 
    $$
    \Bun_{G,D}(X)(S):=\{(\cP, \psi):\cP\in\Bun_G(X)(S), \psi:\cP|_{\operatorname{Spec}\cO_{D}\times S}\simeq G_{\operatorname{Spec}{\cO_D}\times S}\}
    $$
    It is equipped with a map $p_{D}:\Bun_{G,D}(X)\to \Bun_{G}(X)$ sending $(\cP,\psi)$ to $\cP$.
\end{definition}

On the automorphic side, let $K(D)$ be the kernel of the projection map $G(\cO)\to G(\cO_{D})$. Then
\begin{equation*}
    G(F)\backslash G(\AA)/K(D)\simeq \Bun_{G,D}(X)(\FF_q).
\end{equation*}

The map is similar to the one in the unramified case. Choose $(\cP, \phi)\in\Bun_{G,D}(X)(\FF_q)$. As before, we choose trivializations $\phi_y:\cP|_{\cO_y}\simeq G_{\cO_y}$ for any $y\in X$, but we require that $\phi_{y}|_{\cO_{y}/\pi_{y}^{d_y}}=\psi$. We prove that such extensions exist and form a $K(D)'$-torsor, where 
$$
K(D)':=\ker\left(\prod_{y}G(\cO_{y})\to G(\cO_D)\right).
$$
Choose some trivialization $\tau:\cP|_{\prod_{i}\cO_{x_i}}\simeq G_{\prod_{i}\cO_{x_i}}$ and let $h\in G(\cO_D)$ be the element giving the automorphism $\psi\circ(\tau^{-1}|_{\cO_D})$. Then the choices for $\phi_{x_i}$ are in bijection with $g_i\in G(\cO_{x_i})$ such that $g_i\equiv h\pmod{\pi_{x_i}^{d_i}}$. Because of smoothness of $G$ and completeness of $\cO_x$, the maps $G(\cO_{x_i})\to G(\cO_{x_i}/\pi_{x_i}^{d_i})$ are surjective. This proves our claim.

With this choice, we just got an element in $G(F)\backslash G(\AA)/K(D)$, and therefore a diagram
\begin{center}
    \begin{tikzcd}
{G(F)\backslash G(\AA)/G(\cO)}         & \Bun_G(X)(\FF_q)\arrow[l, "\sim"']           \\
G(F)\backslash G(\AA)/K(D) \arrow[u,"p_{D}"'] & {\Bun_{G,D}(X)(\FF_q)} \arrow[l, "\sim"'] \arrow[u,"p_{D}"']\fullstopbelow
\end{tikzcd}
\end{center}
Moreover, we see that the map $p_{D}$ of stacks is a torsor over $K(D)'$ considered as a group scheme over $\FF_p$ in a suitable way.

\subsection{Hecke correspondences for $\GL_n$} \label{Hecke Corres.GLn}
We use the notation from Section \ref{s:preliminaries_on_Hecke_graphs}. For $\GL_n$, there is an equivalence between principal $G$-bundles and vector bundles of rank $n$. Therefore, we will utilize the language of vector bundles. Let $\omega_r$ be the $r$-th fundamental coweight of $\GL_n$ and $\Phi_{D,x}^{\omega_r}$ be the Hecke operator corresponding to it. Note that $\Delta_x = \mathrm{diag}(\pi_x I_r, I_{n-r})$ and $\Delta_y=\mathrm{id}$ for $y\ne x$. For convenience, we set $K:=K(D)$.

\begin{proposition}\label{p:Hecke_correspondence_for_vector_bundles}
    Choose $(\cE,a)\in\mathrm{Bun}_{G, D}(\FF_q)$. Edges $(\cE,a)\to(\cE',b)$ in $\mathrm{Bun}_{G, D}(\FF_q)$ for the Hecke operator $\Phi_x^{\omega_r}$ are in bijection with equivalence classes of exact sequences
$$
\{(\cE',b),\,0\to \cE'\xrightarrow{f}\cE\to k_x^{\oplus r}\to 0\}/\Aut(\cE')\times\GL(k_x^r)
$$
 such that 
    \begin{itemize}
        \item For certain (equivalently, any) $\cO_x$-linear lifts $\tilde{a}: \cE_x \xrightarrow{\sim} \cO_x^n$ and $\tilde{b}: \cE'_x \xrightarrow{\sim} \cO_x^n$ of the level structures $a$ and $b$, the localization $f_x$ satisfies the compatibility condition
$$
    \tilde{a} \circ f_x \circ \tilde{b}^{-1} \in K_{x} \Delta_x K_{x}.
$$
    \item For each $y\ne x$, $f|_{\cO_{d_y[y]}}\circ b|_{\cO_{d_y[y]}}=a|_{\cO_{d_y[y]}}$. In particular, $b|_{\cO_D}$ is uniquely determined by $f$ and $a$.
    \end{itemize}
\end{proposition}
\begin{proof}
The Hecke modification of bundles gives the desired exact sequence. Since the support of $\mathcal{E} / \mathcal{E}'$ is exactly $\{x\}$, the morphism $f$ induces an isomorphism of vector bundles over $X \setminus x$. 

Since the support of $D'$ is entirely contained in $X \setminus x$, the restriction of $f$ to the formal neighborhood of any $y \in \operatorname{supp}D\setminus x$ is an isomorphism of free $\mathcal{O}_{d_y[y]}$-modules.

Decompose $K\Delta K=\sqcup_i \tau_i\Delta K$ into right $K$-cosets. In the bijection described above, the pair $(\cE,a)$ is given by trivializations $(\phi_y)_y$ with $\phi_{y}$ lifting $\tilde a|_{\cO_{d_y[y]}}$. They give $g=(\phi_\eta\phi_y^{-1})_y\in G(\AA)$. Hecke correspondence asserts that the neighbors of $g$ in the Hecke graph are precisely $g\tau_i\Delta$. In terms of trivializations, the family $(\phi_y)_y$ is connected to $(\phi_y')_y$ with $\phi_{y}'^{-1}=\phi_{y}^{-1}\tau_i\Delta$ for $y\in\operatorname{supp}D$ and $\phi_y'=\phi_y$ for $y\notin\operatorname{supp}D$. Again, we treat $\phi'_y$ as lifting $\tilde b$. The first condition can be rewritten as
$$
\Delta_y\circ \phi_y'=\tau_i^{-1}\circ\phi_y,\quad y\in \operatorname{supp} D.
$$

By the unramified case, $\tau_i^{-1}|_{\cO_y}$ is precisely the map $f_y$. If $y=x$, this gives the precise condition stated. If $y\ne x$, then $\Delta_y=1$ and there is a unique solution for $\phi_y'$ over $\cO_{y}$ and over $\cO_{d_y[y]}$, given by the desired condition. Therefore, we can restrict modulo $\pi_{y}^{d_y}$ in this case, so we are done.
\end{proof}

\begin{theorem}\label{t:ramified_Hecke_correspondences}
 For a point $(\cE, a) \in \mathrm{Bun}_{G, d[x]}(\FF_q)$, where $a: \cE_x \otimes (\cO_x/\pi_x^d) \xrightarrow{\sim} (\cO_x/\pi_x^d)^n$ is a level structure and $d\ge 1$, we have the following: 
\begin{enumerate}[label=(\roman*)]
    \item For all edges $(\cE,a)\to(\cE', b)$,  the subbundle $\cE' \subset \cE$ is unique, independent of the outgoing edge, and satisfies $\cE/\cE'\simeq k_x^{\oplus r}$.
    \item The double coset $K_x\Delta_x K_x$ is decomposed into left $K$-cosets as
    $$  
    K_x\Delta_x K_x=\bigsqcup_{C}\mat{I_r}{\pi_x^{d}C}0{I_{n-r}}\Delta_x K_x=\bigsqcup_{C}\Delta_x\mat{I_r}{\pi_x^{d-1}C}0{I_{n-r}} K_x,
    $$
    where $C$ runs through matrices in $M_{r \times (n-r)}(k_x)$. In particular, the outgoing degree of each vertex equals $|k_x|^{r(n-r)}$.
\end{enumerate}
\end{theorem}

\begin{proof}
Let $\tilde{a}: \cE_x \xrightarrow{\sim} \cO_x^n$ be an $\cO_x$-linear lift of the level structure $a$. A Hecke modification assigns an inclusion of $\cO_x$-modules $f: \cE'_x \hookrightarrow \cE_x$ and a new lifted level structure $\tilde{b}: \cE'_x \xrightarrow{\sim} \cO_x^n$ such that the matrix representing $f$ satisfies:
\begin{equation*}
    \tilde{a} \circ f \circ \tilde{b}^{-1} = H \in K_x \Delta_x K_x.
\end{equation*}
The image of $f$ inside $\cE_x$ corresponds under $\tilde{a}$ to a lattice $N \subset \cO_x^n$. By definition, $N = H \cO_x^n$. Since $H$ lies in the double coset, we may write $H = k_1 \Delta_x k_2$ for some $k_1, k_2 \in K_x$. Because $k_2 \in \mathrm{GL}_n(\cO_x)$, it stabilizes the standard lattice, hence $N = k_1 \Delta_x \cO_x^n$. 

We claim that $N = \Delta_x \cO_x^n$ for all $k_1 \in K_x$. To verify this, we compute the conjugation $\Delta_x^{-1} k_1 \Delta_x$. Write $k_1 = I_n + \pi_x^d X$ for some matrix $X \in M_n(\cO_x)$. Then:
\begin{align*}
    \Delta_x^{-1} k_1 \Delta_x &= I_n + \pi_x^d \begin{pmatrix} \pi_x^{-1} I_r & 0 \\ 0 & I_{n-r} \end{pmatrix} \begin{pmatrix} X_{11} & X_{12} \\ X_{21} & X_{22} \end{pmatrix} \begin{pmatrix} \pi_x I_r & 0 \\ 0 & I_{n-r} \end{pmatrix} \\
    &= I_n + \begin{pmatrix} \pi_x^d X_{11} & \pi_x^{d-1} X_{12} \\ \pi_x^{d+1} X_{21} & \pi_x^d X_{22} \end{pmatrix}\in\GL(\cO_x).
\end{align*}
This implies that $k_1 \Delta_x \cO_x^n = \Delta_x \cO_x^n$. The $\cO_x$-lattice $N$ is therefore invariant under the action of the entire double coset. We then have $\cE'_x=\tilde a^{-1}(\Delta_x\cO_x^n)\subset \cE_x$, and via this identification, $f$ is the inclusion map. By the Beauville-Laszlo theorem, the category of vector bundles on $X$ is equivalent to the category of gluing data $(\cE'_{X \setminus x}, \cE'_x, \beta)$, where $\cE'_{X \setminus x}$ is a vector bundle on the punctured curve, $\cE'_x$ is a free $\cO_x$-module of rank $n$ on the formal disk $D_x = \mathrm{Spec}(\cO_x)$, and $\beta: \cE'_{X \setminus x} \otimes_{\cO_{X \setminus x}} F_x \xrightarrow{\sim} \cE'_x \otimes_{\cO_x} F_x$ is an isomorphism over the formal punctured disk $D_x^\times = \mathrm{Spec}(F_x)$. This proves that the subbundle $\cE'$ is uniquely determined. 

To prove the second part, we use the isomorphism
$$
K_x/(\Delta_xK_x\Delta_x^{-1}\cap K_x)\simeq K_x\Delta_x K_x/K_x
$$
sending the trivial coset to $\Delta_xK_x/K_x$. Taking the transpose of the above computation, we get that
$$
\Delta_xK_x\Delta_x^{-1}\cap K_x=I_n + \begin{pmatrix} \pi_x^d X_{11} & \pi_x^{d+1} X_{12} \\ \pi_x^{d} X_{21} & \pi_x^d X_{22} \end{pmatrix}.
$$
For $C\in M_{r \times (n-r)}(k_x)$, we have
\begin{align*}
    \mat{I_r}{\pi_x^{d}C}0{I_{n-r}}(\Delta_xK_x\Delta_x^{-1}\cap K_x)&=I_n + \begin{pmatrix} \pi_x^d (X_{11}+\pi_x^dCX_{21}) & \pi_x^{d+1} X_{12}+\pi_x^d C(1+\pi_x^dX_{22}) \\ \pi_x^{d} X_{21} & \pi_x^d X_{22} \end{pmatrix}\\
    &=I_n + \begin{pmatrix} \pi_x^d X_{11}' & \pi_x^d C+\pi_x^{d+1} X_{12}' \\ \pi_x^{d} X_{21} & \pi_x^d X_{22} \end{pmatrix},
\end{align*}
where $X_{11}':=X_{11}+\pi_x^dCX_{21}$ and $X_{12}':=X_{12}+\pi_x^{d-1}CX_{22}$. These cosets are clearly disjoint and cover all $K_x$, so we are done.

\end{proof}

We finish this section with a description of the way to transfer from $G=\GL_n$ to $\bar G=\PGL_n$. Let $Z$ be the center of $G$. As before, assume ramification at a divisor $D=D'+d[x]$ and denote $K:=K(D)$.
\begin{proposition}\label{p:from_GL_to_PGL}
Let $\Gamma_G$ and $\Gamma_{\bar G}$ be the Hecke graphs for $G$ and $\bar G$ associated with $\Phi_x^{\omega_r}$. Then
\begin{enumerate}[label=(\roman*)]
    \item The quotient of the set of vertices of $\Gamma_G$ by the action of the abelian group $Z(F)\backslash Z(\AA)/(K\cap Z(\AA))$ is precisely $\bar G(F)\backslash \bar G(\AA)/\bar K=\Gamma_{\bar G}$. Geometrically, vertices of $\Gamma_{\bar G}$ are vector bundles with a level structure at $D'$ up to tensoring with a line bundle with a level structure at $D'$.
    \item For $a,b\in \Gamma_{\bar G}$, fix a lift $\tilde a$ of $a$ to $\Gamma_G$. Then the projection between $\Gamma_G$ and $\Gamma_{\bar G}$ gives a bijection between the set of edges $a\to b$ in $\Gamma_{\bar G}$ and the set of edges between $a$ and all possible lifts of $b$ to $\Gamma_G$. In other words, $\Gamma_{\bar G}$ is obtained from $\Gamma_G$ by quotient of the set of vertices by the action of $Z(F)\backslash Z(\AA)/(K\cap Z(\AA))$ while keeping the edges.
\end{enumerate}
\end{proposition}

\begin{proof}
The first item is clear. To prove $(ii)$, recall that after decomposing $K\Delta K$ into left cosets $\bigsqcup_i\tau_i\Delta K$, we get that all neighbors of $g\in \Gamma_G$ are $G(F)g\tau_i K$, counted with multiplicities. Same applies to $\bar G$. Therefore, item $(ii)$ is equivalent to the map
$$
K\Delta K/K\to K\Delta KZ(\AA)/KZ(\AA)
$$
being bijective. Equivalently,
$$
\Delta^{-1} K\Delta Z(\AA)\cap K=\Delta^{-1} K\Delta\cap K.
$$

Since $Z(\AA)\subset \prod_{y\in |X|}Z(F_y)$, it is sufficient to prove that
$$
\Delta_y^{-1} K_y\Delta_y Z(F_y)\cap K_y=\Delta_y^{-1} K_y\Delta_y\cap K_y,\qquad \forall_{y\in |X|}.
$$
For $y\ne x$, this is evident since $\Delta_y=\mathrm{id}$ in this case. Thus, we are left with the case $y=x$.

Take $\gamma \in \Delta_x K_x\Delta_x^{-1} Z(F_x)\cap K_x$. Write
$$
    \gamma = \Delta^{-1}_x p \Delta_x z,\qquad p\in K_x,z\in Z(F_x).
$$
It is enough to prove that $z \in K_x$, which means that $z\in Z(\cO_x)$ and $z \equiv I_n \pmod{\pi_x^d}$. 

Write the matrix $p \in K_x$ in block form corresponding to the partition $(r, n-r)$:
$$
    p = \begin{pmatrix} A & B \\ C & D \end{pmatrix}.
$$
Writing $z=\zeta\cdot\mathrm{id}$ for $\zeta\in F_x^\times$, we compute
$$
    \gamma=\Delta^{-1} p \Delta z= \begin{pmatrix} \pi_x^{-1} I_r & 0 \\ 0 & I_{n-r} \end{pmatrix} \begin{pmatrix} A & B \\ C & D \end{pmatrix} \begin{pmatrix} \pi_x\zeta I_r & 0 \\ 0 & \zeta I_{n-r} \end{pmatrix} = \begin{pmatrix} A\zeta & \pi_x^{-1} B\zeta \\ \pi_x C\zeta & D\zeta \end{pmatrix}.
$$
Since $\gamma,p\in K_x$, we must have
$$
    A, A\zeta \equiv I_r \pmod{\pi_x^d}.
$$
This implies that $\zeta \equiv 1 \pmod{\pi_x^d}$, and therefore $z\in K_x$, as desired.
\end{proof}

\subsection{Hecke correspondences for general $G$}\label{Hecke Corres.G} We use the notation from Section \ref{s:preliminaries_on_Hecke_graphs}. We describe the Hecke operator $\Phi_{D, x}^\mu$ in adelic and geometric languages. Since the proofs are similar to the $\GL_n$ case, we will omit them.

\begin{proposition}\label{t:general_hecke_correspondence}
Fix a vertex $(\mathcal{P}, \psi) \in \Bun_{G,D}(k)$. The directed edges $(\mathcal{P}, \psi) \rightarrow (\mathcal{P}', \psi')$ in the Hecke graph $\Gamma_{D, \mu, x}$ of $\Phi_{D, x}^\mu$ are in bijection with the set of equivalence classes of pairs $(\mathcal{P}', \beta)$, where $\mathcal{P}'$ is a principal $G$-bundle on $X$ and $\beta: \mathcal{P}|_{X \setminus x} \xrightarrow{\sim} \mathcal{P}'|_{X \setminus x}$ is an isomorphism, such that:
\begin{enumerate}[label=(\roman*)]
    \item For any trivializations $\tilde{\psi}_x, \tilde{\psi}'_x$ of $\mathcal{P}|_{\mathcal{O}_x},\mathcal{P}'|_{\mathcal{O}_x}$ lifting $\psi_x,\psi_x'$, respectively, the local transition element $\tilde{\psi}_x \circ \beta^{-1} \circ (\tilde{\psi}'_x)^{-1} \in G(F_x)$ lies in the double coset $K(D) \Delta^\mu K(D)$.
    \item The level structure is strictly preserved away from the modification point, meaning $\psi'_y \circ \beta|_{D'} = \psi_y$ for any $y\in\operatorname{supp}D'$.
\end{enumerate}

\end{proposition}

\section{The $\GL_n$ case} \label{GLn case}
\subsection{Geometry of Hecke graphs in the cusp of $\Bun_G$}
 Let $G=\GL_n$ for $n\ge 2$. We continue using notation from Section \ref{s:preliminaries_on_Hecke_graphs}. Fix a decomposition $D=D_1+D_2$ into a sum of effective $D_1$ and $D_2$ satisfying
\begin{equation}\label{eq:d1_d2_condition}
    \operatorname{supp}(D_1)\cap\operatorname{supp}(D_2)=\varnothing,
\qquad
x\notin\operatorname{supp}(D_1).
\end{equation}
We have a map
$$
p_{D,D_2}:\Bun_{D}(X)(\FF_q)\to \Bun_{D_2}(X)(\FF_q)
$$
forgetting the level structure outside $D_2$.

Our goal is to understand the graph of $\Phi_{D,x}^{\omega_r}$ on the domain in terms of the graph on the codomain of $p_{D,D_2}$. For this, we recall the Harder-Narasimhan stratification. We follow the exposition in \cite[Section 5]{Hei10}.
\begin{definition}
    Let $\cE$ be a rank $n$ vector bundle. Its {\bfseries degree}, denoted by $\deg E$, is defined as the degree of the corresponding {\bfseries determinant line bundle} $\det \cE:=\bigwedge^n \cE$. For properties of the degree, see  \cite[\href{https://stacks.math.columbia.edu/tag/0AYQ}{0AYQ}]{Stacks}.

    The {\bfseries slope} of $\cE$ is defined as
    $$
    \mu(\cE):=\frac{\deg \cE}{\rk \cE}.
    $$

    We say that a vector bundle $\cF$ is a {\bfseries subbundle} of a vector bundle $\cE$ if $\cF\subset \cE$ and $\cE/\cF$ is a vector bundle. A vector bundle $\cE$ is called {\bfseries (semi)stable} if for any vector subbundle $\cF\subset \cE$ we have $\mu(\cF)<\mu(\cE)$ ($\mu(\cF)\le\mu(\cE)$).
\end{definition}

\begin{remark}\label{r:semistability_with_embedded_bundles}
    {\rm
        Let $\cE_1\hookrightarrow\cE_2$ be an embedding of vector bundles. Let $D$ be the divisor of the torsion part of $\cE_2/\cE_1$. Then $\cE_1(D)\subset\cE_2$ is a subbundle, and we have 
        $$
        \deg\cE_1(D)=\deg \cE_1+\deg D\cdot\rk\cE_1\ge \deg\cE_1,
        $$ 
        and hence
        $$
        \mu(\cE_1(D))=\mu (\cE_1)+\deg D\ge \mu(\cE_1).
        $$
        In particular, if $\cE_2$ is semistable, then $\mu(\cE_1)\le\mu(\cE_2)$.
    }
\end{remark}

\begin{theorem}[Harder-Narasimhan filtration]\label{t:HN_filt}
    For any vector bundle $\cE$ on $X$, there exists a unique filtration
    $$
    0=\cE_0\subset \cE_1\subset\ldots\subset \cE_k=\cE
    $$
    of subbundles such that
    \begin{itemize}
        \item [(i)] The bundles $\cE_i/\cE_{i-1}$ are semistable,
        \item [(ii)] $\mu(\cE_1/\cE_0)>\mu(\cE_2/\cE_1)>\ldots>\mu(\cE_k/\cE_{k-1})$.
    \end{itemize}
\end{theorem}
\begin{definition}
    A {\bfseries Harder-Narasimhan (HN) filtration} on a vector bundle $\cE$ is the filtration from \ref{t:HN_filt}.

    A {\bfseries Harder-Narasimhan (HN) type} of $\cE$ is the sequence of points $(\rk \cE_i,\deg \cE_i)$ in $\RR^2$, where $\cE_\bullet$ is the HN stratification of $\cE$. By definition, these points form a convex polygon in $\RR^2$.

    Fix a $HN$ type $(d_i,r_i)$. The corresponding {\bfseries Harder-Narasimhan (HN) stratum} $\Bun_G^{(r_i,d_i)}$ is the stack whose $S$-points are all vector bundles $\cE$ on $X\times S$ such that for any $s\in S$, the restriction $\cE|_{X\times \{s\}}$ has HN type $(r_i,d_i)$.
\end{definition}

Now, let us state a sequence of useful lemmas.
\begin{lemma}\label{l:slopes_in_exact_seq}
    Let 
    $$
    0\to\cE_1\to\cE_2\to\cE_3\to0
    $$
    be an exact sequence of vector bundles. Then
    \begin{itemize}
        \item [(i)] $\mu(\cE_1)>\mu(\cE_2)$ iff $\mu(\cE_1)>\mu(\cE_3)$ iff $\mu(\cE_2)>\mu(\cE_3)$.
        \item [(ii)] The same holds with $>$ replaced by $\ge,\le,<$.
        \item [(iii)] $\min\{\mu(\cE_1), \mu(\cE_3)\}\le \mu(\cE_2) \le \max\{\mu(\cE_1), \mu(\cE_3)\}$.
    \end{itemize}
\end{lemma}
\begin{proof}
    Let $r_i$ be the ranks and $d_i$ be the degrees of the corresponding bundles. Exactness implies that $r_2=r_1+r_3$ and
    $$
    \det \cE_2\simeq \det \cE_1\otimes \det \cE_3, 
    $$
    and therefore $d_2=d_1+d_3$. Then
    $$
    \mu(\cE_2)=\frac{d_1+d_3}{r_1+r_3},
    $$
    and it is easy to check the first and second statement with $\mu(E_i)$ replaced by $\frac{d_1}{r_1}$, $\frac{d_1+d_3}{r_1+r_3}$, and $\frac{d_3}{r_3}$. The last part follows from $(i)$ and $(ii)$.
\end{proof}

\begin{lemma}\label{l:homs_from_high_to_low}
    Let $\cE_1$ and $\cE_2$ be semistable bundles with $\mu(\cE_1)>\mu(\cE_2)$. Then
    $$
    \Hom(\cE_1,\cE_2)=0.
    $$
    In particular, $H^0(X,\cE)=0$ for any semistable bundle $\cE$ with $\mu(\cE)<0$.
\end{lemma}
\begin{proof}
    Suppose that there is a nonzero $f\in \Hom(\cE_1,\cE_2)$, and consider two exact sequences
    \begin{align*}
        0&\to\Ker f\to \cE_1\to\Im f\to0,\\
        0&\to\Im f\to \cE_2\to\Coker f\to0.
    \end{align*}
    Because $\Im f\subset \cE_2$ and $\ker f\subset \cE_1$, they are torsion-free, and hence are vector bundles. By Remark \ref{r:semistability_with_embedded_bundles}, $\Ker f\ne 0$, because otherwise $f$ would embed $\cE_1$ to $\cE_2$, contradicting the assumption $\mu(\cE_1)> \mu(\cE_2)$ and semistability of $\cE_2$. Since $\cE_1$ is semistable, we have $\mu(\Ker f)\le\mu(\cE_1)$. Applying Lemma \ref{l:slopes_in_exact_seq} and semistability of $\mu(\cE_2)$, we get
    $$
    \mu(\cE_1)\le \mu(\Im f)\le \mu(\cE_2),
    $$
    a contradiction to our assumption. This proves the statement.

    The last part is implied by the observation that $H^0(X,\cE)\simeq \Hom(\cO_X,\cE)$, and $\cO_X$ is semistable (as any line bundle) of slope $0$.
\end{proof}

\begin{lemma}\label{l:HN_functoriality}
    
    Let $\cE$ be a vector bundle with HN filtration $\cE_\bullet$. Then any morphism $f\in\End(\cE)$ preserves the filtration, i.e. $f(\cE_i)\subset \cE_i$ for any $i$.
\end{lemma}
\begin{proof}
    Let $k$ be the smallest integer such that $f(\cE_1)\subset \cE_k$. Then the induced map $f:\cE_1\to\cE_k/\cE_{k-1}$ is non-zero, which by Lemma \ref{l:homs_from_high_to_low} implies that $k\le 1$, thus $f(\cE_1)\subset \cE_1$. By induction hypothesis applied to $\cE/\cE_1$, we get that $f(\cE_k)\subset \cE_k$ for all $k>1$ as well, which finishes the proof.
\end{proof}

In particular, this implies that any map $f:\cE\to\cE$ induces a map $\gr f:\gr\cE_\bullet\to\gr\cE_\bullet$. Therefore, we have an embedding
$$
\End(\cE)\hookrightarrow \End(\gr\cE_\bullet).
$$

\begin{proposition}\label{p:structure_of_very_unstable_bundles}
    Let $\cE$ be a vector bundle with HN filtration $\cE_\bullet$ satisfying 
    $$
    \mu(\cE_i/\cE_{i-1})-\mu(\cE_{i+1}/\cE_{i})>2g-2
    $$
    for any $i$ for which the above quotients are nonzero. Then
    $$
    \cE\simeq \gr\cE_\bullet.
    $$
    In particular,
    $$
    \End(\cE)\simeq \End(\gr\cE_\bullet)\simeq\bigoplus_{i\ge j}\Hom(\cE_i/\cE_{i-1},\cE_j/\cE_{j-1}).
    $$
\end{proposition}
\begin{proof}
    It is sufficient to show that any short exact sequence
    $$
    0\to \cE_i/\cE_{i-1}\to \cE_{i+1}/\cE_{i-1}\to \cE_{i+1}/\cE_{i}\to 0
    $$
    splits. Its splitting is equivalent to the corresponding extension class in $\Ext^1_{\cO_X}(\cF_{i+1},\cF_i)$ to be zero, where we denoted $\cE_i/\cE_{i-1}$ by $\cF_i$ for simplicity. Using Serre duality, we compute
    \begin{align*}
        \Ext^1_{\cO_X}(\cF_{i+1},\cF_i)&\simeq H^1(X,\cF_i\otimes \cF_{i+1}^\vee)\\
        &\simeq H^0(X,\cF_i^\vee\otimes \cF_{i+1}\otimes \omega_X)^*\\
        &=\Hom(\cO_X,\cF_i^\vee\otimes \cF_{i+1}\otimes \omega_X)^*.
    \end{align*}

    We have
    $$
    \mu(\cF_i^\vee\otimes \cF_{i+1}\otimes \omega_X)=-\mu(\cF_i)+\mu(\cF_{i+1})+2g-2<0,
    $$
    which implies that the degree of this bundle is negative, and hence
    $$
    \Hom(\cO_X,\cF_i^\vee\otimes \cF_{i+1}\otimes \omega_X)^*=0.
    $$
    This finishes the proof.
\end{proof}

Write $D=D_1+D_2$ for effective divisors $D_1$ and $D_2$ with $\supp D_1\cap\supp D_2=\varnothing$, and recall that $\cO_D$ is the ring of functions on $D$.

\begin{theorem}\label{t:gln_image_of_aut_group}
    Let $\cE$ be a vector bundle with HN filtration $\cE_\bullet$, with the conditions that the subquotients $\cL_i:=\cE_i/\cE_{i-1}$ are line bundles and $\deg\cL_i-\deg\cL_{i+1}>2g-2+\deg D$. Let also $\phi_{D_2}$ be a level structure at $D_2$ (if $D_2=0$, we mean that there is no level structure). Then 
    \begin{itemize}
        \item [(i)] $\cE\simeq \bigoplus_{i}\cL_i$,
        \item [(ii)] The image of $\Aut((\cE,\phi_{D_2}))$ in $\Aut(\cE|_{\cO_{D_1}})\simeq \GL_n(\cO_{D_1})$ under a trivialization of $\cE$ over $\cO_{D_1}$ sending $\cL_i$ to the $i$-th coordinate line of $\cO_{D_1}^{\oplus n}$ coincides with $T(k)\ltimes U(\cO_{D_1})$ if $D_2=0$ and $U(\cO_{D_1})$ if $D_2\ne 0$.
    \end{itemize}
\end{theorem}

\begin{proof}
    The first part of the Theorem follows immediately from \ref{p:structure_of_very_unstable_bundles}. From the second part of the same theorem, it follows that
    $$
    \Aut(\cE)\simeq \bigoplus_i \Aut(\cL_i)\oplus \bigoplus_{i>j}H^0(X,\cL_i\otimes\cL_j^\vee),
    $$
    where the multiplication is given by visualising the above expression in the upper-triangular matrix form:

    \[
\renewcommand{\arraystretch}{1.8} 
\setlength{\arraycolsep}{6pt}     
\begin{pmatrix}
    k^\times & H^0(\cL_1\otimes\cL_2^\vee) & H^0(\cL_1\otimes\cL_3^\vee) & \cdots & H^0(\cL_1\otimes\cL_n^\vee) \\
    0 & k^\times & H^0(\cL_2\otimes\cL_3^\vee) & \cdots & H^0(\cL_2\otimes\cL_n^\vee) \\
    \vdots & \vdots & \ddots & \cdots & \vdots \\
    0 & 0 & 0 & \ddots & H^0(\cL_{n-1}\otimes\cL_n^\vee) \\
    0 & 0 & 0 & \cdots & k^\times
\end{pmatrix}.
\]

In particular, the image of $\Aut(\cE)$ in $\Aut(\cE|_{\cO_D})$ equals
\[
\renewcommand{\arraystretch}{1.8} 
\setlength{\arraycolsep}{6pt}     
\begin{pmatrix}
    k^\times & H^0(\cL_1\otimes\cL_2^\vee)_{D} & H^0(\cL_1\otimes\cL_3^\vee)_{D} & \cdots & H^0(\cL_1\otimes\cL_n^\vee)_{D} \\
    0 & k^\times & H^0(\cL_2\otimes\cL_3^\vee)_{D} & \cdots & H^0(\cL_2\otimes\cL_n^\vee)_{D} \\
    \vdots & \vdots & \ddots & \cdots & \vdots \\
    0 & 0 & 0 & \ddots & H^0(\cL_{n-1}\otimes\cL_n^\vee)_{D} \\
    0 & 0 & 0 & \cdots & k^\times
\end{pmatrix},
\]
where by $H^0(\cL)_{D}$ we denote the image of the restriction map
$$
H^0(X,\cL)\to H^0(X,\cL|_{D})=H^0(D,\cL|_{D}).
$$
To compute these images, we apply the long exact sequence for cohomology to
$$
0\to\cL(-D)\to\cL\to \cL|_{D}\to 0,\qquad \cL=\cL_i\otimes \cL_{j}^\vee,i>j
$$
to get
$$
0\to H^0(\cL(-D))\to H^0(\cL)\to H^0(\cL|_{D})\to H^1(\cL(-D)).
$$
Our assumption that $\deg\cL_i-\deg\cL_j>2g-2+\deg D$ implies that $H^1(\cL(-D))=0$. Therefore, $H^0(\cL)_{D}=H^0(\cL|_{D})\simeq \cO_{D}$. This shows that the image of $\Aut(\cE)$ in $\Aut(\cE|_{\cO_{D}})$ equals $T(k)\ltimes U(\cO_{D})$. Now, $\Aut((\cE,\phi_{D_2}))$ is the preimage of 
$$
\Aut(\cE|_{\cO_{D_1}})\times\operatorname{id}\subset \Aut(\cE|_{\cO_{D_1}})\times \Aut(\cE|_{\cO_{D_2}})=\Aut(\cE|_{\cO_{D}}).
$$
in $\Aut(\cE)$. This shows that 
$$
(T(k)\ltimes U(\cO_{D}))\cap (G(\cO_{D_1})\times\operatorname{id})=
\begin{cases}
    T(k)\ltimes U(\cO_{D_1}),&D_2=0,\\
    U(\cO_{D_1}),&D_2\ne 0
\end{cases}
$$ 
is the image of $\Aut((\cE,\phi_{D_2}))$ in $\Aut(\cE|_{\cO_{D}})$. Projection to $\cO_{D_1}$ finishes the proof.
\end{proof}

The next corollary generalizes and proves \cite[Conjecture 5.16]{AB24}.
\begin{corollary}\label{cor:gln_preimage_size}
    Let $k=\FF_q$ and $\cE$ be a vector bundle satisfying the assumptions of Theorem \ref{t:gln_image_of_aut_group}. Then 
    $$
    \# p_{D,D_2}^{-1}((\cE,\phi_x))=|T(\cO_{D_1})|/|T(\FF_q)|^{\delta_{D_2,0}}\cdot|\operatorname{Fl}_n(\cO_{D_1})|,
    $$
    where $\operatorname{Fl}_n$ is the flag variety for $\GL_n$ and $\delta$ is the Kronecker delta-function. More explicitly,
    $$
    \# p_{D,D_2}^{-1}(\cE)=\frac{q^{(\deg D_1-\deg (D_1)_{\mathrm{red}})\frac{n(n+1)}2}}{(q-1)^{n\delta_{D_2,0}}}\prod_{y\in\operatorname{supp}D_1}\prod_{k=1}^n(q^{k\deg y}-1),
    $$
    where $(D_1)_{\mathrm{red}}:=\sum_{y\in\operatorname{supp}D_1}1\cdot[y]$.
\end{corollary}

\begin{proof}
    This is a direct consequence of Theorem \ref{t:gln_image_of_aut_group} and a standard computation of the number of elements in $\GL_n$.
\end{proof}

 Now, let us study the graph of the Hecke operator $\Phi_x^{w_k}$ for $k=1\ldots n$ close to the cusp of $\Bun_{G,D}$. We will use results of Section \ref{Hecke Corres.GLn}. Let us choose two rank $n$ vector bundles 
 $$
 (\cE,\phi),(\cE',\phi')\in\Bun_{G,D}
 $$
 satisfying the assumptions of Theorem \ref{t:gln_image_of_aut_group}. Let $\cE_\bullet$ and $\cE_\bullet'$ be the corresponding HN filtrations and $\cL_i$ and $\cL_i'$ be the corresponding subquotients. We want to compute the number of arrows $(\cE,\phi)\to(\cE',\phi')$, equivalently, the number of isomorphism classes of extensions from Proposition \ref{p:Hecke_correspondence_for_vector_bundles}.

\begin{proposition}\label{p:cusp_connected_to_cusp}
Let in addition $\cE$ satisfy the condition that $\deg\cL_{i}>\deg\cL_{i+1}+\deg x$. Then
\begin{itemize}
    \item[(i)] $f^{-1}(\cE_i)=\cE_i'$.
    \item[(ii)] All $\cL_i'$ are line bundles and $f$ identifies $\cL_i'$ with $\cL_i(-p_ix)$ for some $0\le p_i\le 1$ such that $\sum_i p_i=r$. In particular,
$$
\deg \cL_i'\le\deg\cL_i\le\deg\cL_i'+\deg x.
$$
\end{itemize}
\end{proposition}

\begin{proof}
    Consider the filtration $\cE_i'':=f^{-1}(\cE_i)$ on $\cE'$, and let $\cL_i'':=\cE_i''/\cE_{i-1}''$. Then, because $\cE$ and $\cE'$ differ only at $x$ by a finite length sheaf, for each $i$ there is a nonnegative integer $p_i$ such that
    \begin{align*}
        0&\to\cL_i''\xrightarrow{f}\cL_i\to k_x^{\oplus p_i}\to0,\\
        0&\to\cE_i''\xrightarrow{f}\cE_i\to k_x^{\oplus (p_1+\ldots+p_i)}\to0.
    \end{align*}

This implies that $f(\cL_i'')=\cL_i(-p_ix)$, and therefore $\deg\cL_i''=\deg\cL_i-p_i\deg x$. Moreover, since $\cL_i/\cL_i(-p_ix)\simeq \cO_x/\pi_x^{p_i}$ is not isomorphic to $k_x^{\oplus p_i}$ as an $\cO_X$-module if $p_i>1$, then we get $0\le p_i\le 1$.

Because we assumed that $\deg\cL_i>\deg\cL_{i+1}+\deg x$, we get
$$
\deg\cL_i''\ge\deg\cL_i-\deg x>\deg\cL_{i+1}\ge \deg\cL_{i+1}'',
$$
which by uniqueness of the HN filtration implies that $\cE_\bullet''$ is the HN filtration of $\cE'$. Therefore, $\cE'_\bullet=\cE_\bullet''$ and $\cL_i'=\cL_i''$, which finishes the proof.
\end{proof}

Because of this proposition, we get the following picture. Let $\cE$ satisfy
$$
\deg\cL_i>\deg\cL_{i+1}+2g-2+\deg D+\deg x.
$$
 Then any neighbor of $\cE$ in the Hecke graph is a map $f\in\Hom(\bigoplus_i\cL_i(-p_ix_i),\bigoplus_i\cL_i)$ for some  partition $r=p_1+\ldots+p_n$ with $0\le p_i\le 1$. This map must be an embedding with cokernel $k_x^{\oplus r}$. This condition is equivalent to the following:
 \begin{itemize}
     \item  The induced maps $\cL_i(-p_ix)\to\cL_i$ are nonzero. Then they are constant multiples of the canonical injections $\iota_{p_ix}$.
     \item The reduction $\bar f$ of $f$ to $k_x$ has rank $n-r$. Since the matrix of $\bar f$ in the bases $\cL_i$, $\cL_i(-p_ix)$ is upper-triangular with zeroes at precisely $r$ diagonal entries, it must have zeroes at the entries $(i,j)$ where $p_i=p_j=1$. This amounts to the statement that the induced maps $\cL_i(-p_ix)\to\cL_j$ factor through $\cL_i\to \cL_j$.
 \end{itemize}

Therefore, all isomorphism classes of extensions from Proposition \ref{p:Hecke_correspondence_for_vector_bundles} are classified by
\begin{itemize}
    \item A partition $r=p_1+\ldots+p_n$ with $0\le p_i\le 1$.
    \item $Y/W$, where $Y\subset \Hom(\bigoplus_i\cL_i(-p_ix_i),\bigoplus_i\cL_i)$ is the set of all maps whose induced maps $\cL_i(-p_ix)\to\cL_i$ are the canonical injections $\iota_{p_ix}$, the maps $\cL_i(-p_ix)\to\cL_j$ factor through $\cL_i\to\cL_j$ for $i\ne j$ with $p_i=p_j=1$, and $W\subset\Aut(\bigoplus_i\cL_i)$ is the set of all automorphisms whose induced maps on $\cL_i\to\cL_i$ are the identity maps.
    
    In other words, it is the quotient of the space
     \[
\renewcommand{\arraystretch}{1.8} 
\setlength{\arraycolsep}{6pt}     
Y=\begin{pmatrix}
    \iota_{p_1x} & H^0(\cL_1\otimes\cL_2^\vee(\delta_{p_1,0}p_2x)) &  \cdots & H^0(\cL_1\otimes\cL_n^\vee(\delta_{p_1,0}p_nx)) \\
    0 & \iota_{p_2x} &  \cdots & H^0(\cL_2\otimes\cL_n^\vee(\delta_{p_2,0}p_nx)) \\
    \vdots & \vdots &  \ddots & \vdots \\
    0 & 0 &  \cdots & H^0(\cL_{n-1}\otimes\cL_n^\vee(\delta_{p_{n-1},0}p_nx)) \\
    0 & 0 &  \cdots & \iota_{p_nx}
\end{pmatrix}
\]
by the group of matrices multiplying on the right
$$
\renewcommand{\arraystretch}{1.8} 
\setlength{\arraycolsep}{6pt}     
W=\begin{pmatrix}
    1 & H^0(\cL_1\otimes\cL_2^\vee(p_2x-p_1x)) &  \cdots & H^0(\cL_1\otimes\cL_n^\vee(p_nx-p_1x)) \\
    0 & 1 &  \cdots & H^0(\cL_2\otimes\cL_n^\vee(p_nx-p_2x)) \\
    \vdots & \vdots &  \ddots & \vdots \\
    0 & 0 &  \cdots & H^0(\cL_{n-1}\otimes\cL_n^\vee(p_nx-p_{n-1}x)) \\
    0 & 0 &  \cdots & 1
\end{pmatrix}.
$$

\end{itemize}

\begin{remark}
    {\rm
        The partitions $r=p_1+\ldots+p_n$ with $0\le p_i\le 1$ can be described in the following way. Recall that the Hecke operator $\Phi^{w_r}$ corresponds to the fundamental coweight $w_r$, which in the standard basis for the diagonal matrices in $\GL_n$ is represented by $(1,...,1,0,...,0)$ with $r$ ones. If in Proposition \ref{p:Hecke_correspondence_for_vector_bundles} we assumed $\tilde{a} \circ f_x \circ \tilde{b}^{-1}\in \sigma\Delta_x\sigma^{-1} G(\cO_x)$, where $\sigma$ is a permutation matrix treated as an element of the Weyl group, then this would correspond to having $(p_1,...,p_n)=\sigma(w_r)$. Note that the degree of each vertex should equal the number of points in
        $$
        G(\cO_x)\Delta_x G(\cO_x)/G(\cO_x)\simeq G(k_x[[t]])t^{w_r} G(k_x[[t]])/G(k_x[[t]])\simeq \mathrm{Gr}(k,n)(k_x).
        $$
    }
\end{remark}

\begin{proposition}\label{p:cusp_connections_degrees}
    The action of $W$ on $Y$ is free. Therefore, the number of edges from $\cE$ to $\cE'$ equals
\begin{align*}
    |Y/W|=|Y|/|W|&=\prod_{i>j}\frac{\#H^0(\cL_j\otimes\cL_i^\vee(\delta_{p_j,0}p_ix))}{\#H^0(\cL_j\otimes\cL_i^\vee(p_ix-p_jx))}\\
    &=q^{\deg x\left(\sum_{i>j}p_j-\frac{r(r-1)}{2}\right)}\\
    &=q^{\deg x\left(\sum_i(n-i)p_i-\frac{r(r-1)}{2}\right)},
\end{align*}
    and the total outgoing degree of $\cE$ equals
    $$
    \sum_{\substack{0\le p_1,\dots,p_n\le 1\\ r=\sum_{i=1}^np_i}}q^{\deg x\left(\sum_i(n-i)p_i-\frac{r(r-1)}{2}\right)}=\binom{n}{r}_{q^{\deg x}}.
    $$
\end{proposition}
\begin{proof}
    Take matrices $A=(a_{ij})\in Y$ and $g=(g_{ij})\in W$ and assume that $A=Ag$. We will prove that $g_{i,i+k}=0$ for all $i$ by induction on $k\ge 1$. Looking at the entries right above the main diagonal, we see that $A=Ag$ implies
    $$
    a_{i,i+1}=a_{i,i}g_{i,i+1}+a_{i,i+1}g_{i+1,i+1}=\iota_{d_ix}g_{i,i+1}+a_{i,i+1}.
    $$
    Since $\iota_{d_ix}$ is injective, this implies that $g_{i,i+1}=0$.

    Now, assume that the induction hypothesis is true for $1,2,\ldots,k$. We then have
    \begin{align*}
        a_{i,i+k+1}=\sum_{i\le j\le i+k+1} a_{i,j}g_{j,i+k+1}&=a_{i,i}g_{i,i+k+1}+a_{i,i+k+1}g_{i+k+1,i+k+1}\\
        &=\iota_{d_ix}g_{i,i+k+1}+a_{i,i+k+1},
    \end{align*}
    which implies that $g_{i,i+k+1}=0$. This finishes the proof of the freeness of the action.

    To prove the second formula, note that
    \begin{align*}
    \prod_{i>j}\frac{\#H^0(\cL_j\otimes\cL_i^\vee(\delta_{p_j,0}p_ix))}{\#H^0(\cL_j\otimes\cL_i^\vee(p_ix-p_jx))}&=\prod_{i>j}\frac{\#H^0(\cL_j\otimes\cL_i^\vee(p_ix))}{\#H^0(\cL_j\otimes\cL_i^\vee(p_ix-p_jx))}\cdot \prod_{\substack{i>j\\ p_i=p_j=1}}\frac{\#H^0(\cL_j\otimes\cL_i^\vee)}{\#H^0(\cL_j\otimes\cL_i^\vee(x))}\\
    &=q^{\deg x\cdot\sum_{i>j}p_j}q^{-\deg x\cdot \frac{r(r-1)}{2}}.
\end{align*}
    
    To prove the last formula, note that the coefficient of $q^{\alpha\deg x}$ on the left-hand side is the number of Young diagrams with $\alpha$ boxes inside an $r\times(n-r)$ rectangle. The desired bijection takes a decomposition $r=\sum p_i$, assigns to it a collection of indices $i_1,\ldots, i_r$ such that $p_{n+1-i_j}=1$, and to this collection assigns a partition
    $$
    \alpha=\sum_{j=1}^r(i_{j}-j).
    $$
    The inverse map takes a partition $\alpha=\sum_{j=1}^n \lambda_j$ with $\lambda_1\le\ldots\le\lambda_r$ and returns the collection $p_i$ with $p_i=1$ if and only if $i=n+1-j-\alpha_j$ for some $j$. The number of such Young diagrams is the number of Schubert cells in $\operatorname{Gr}(r,n)$ isomorphic to $\AA^\alpha$. Therefore, the total sum is the number of points in $\operatorname{Gr}(r,n)(\FF_{q^{\deg x}})$, which is known to be equal to the right-hand side.
\end{proof} 

This describes the Hecke graph of $\Phi_{x}^{\omega_x}$ at the cusp without any ramification. Our next goal is to show that, as soon as we know the graph of $\Phi_{D_2,x}^{\omega_x}$ at the cusp with ramification $D_2$, we know the graph of $\Phi_{D,x}^{\omega_x}$ at the cusp with ramification $D_1+D_2$. We start with the following definition.

\begin{definition}\label{def:covering_map}
    Let $\Gamma_1$ and $\Gamma_2$ be two oriented graphs such that every vertex has finite ingoing and outgoing degree. A {\bfseries covering map} $p:\Gamma_1\to\Gamma_2$ of graphs is a map on vertices such that
    \begin{itemize}
        \item $p$ is surjective on vertices.
        \item For each $v\in\Gamma_1$ and $w\in\Gamma_2$, $p$ induces a bijection between edges $p(v)\to w$ and $v\to p^{-1}(w)$, where we take the union of edges over all preimages of $w$.
        \item For each $v\in\Gamma_1$ and $w\in\Gamma_2$, $p$ induces a bijection between edges $p^{-1}(w)\to v$ and $w\to v$, where we take the union of edges over all preimages of $v$.
    \end{itemize}
\end{definition}

If we think of oriented graphs as one-dimensional CW complexes, then a covering map is a topological covering map preserving the CW structure and orientations of one-dimensional cells. See Figure \ref{fig:graph_covering} for a visualization. In particular, such a map is completely determined by $\Gamma_2$, the set of vertices of $\Gamma_1$, and the way topological loops lift.

\begin{figure}
    \centering
\scalebox{.8}{

\includegraphics{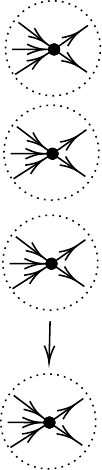}

}
\caption{Covering map of graphs.}
\label{fig:graph_covering}
\end{figure}

Let $\Gamma_{D}:=\Gamma_{D,x}^{\omega_r}$ and $\Gamma_{D_2}:=\Gamma_{D_2,x}^{\omega_r}$. Let 
$$
\Gamma_{D}^{>k}:=\{(\cE,\phi)\in \Gamma_{D}:\rk(\cL_i)=1,\deg\cL_i>\deg\cL_{i+1}+2g-2+k+\delta_{d_x,0}\deg x\}.
$$
Recall that we fixed a decomposition $D=D_1+D_2$ satisfying (\ref{eq:d1_d2_condition}).
\begin{lemma}\label{l:covering_map_GL}
    The map $p_{D,D_2}:\Gamma_{D}^{>k}\to \Gamma_{D_2}^{>k}$ is a covering map of graphs for $k\ge \deg D$.
\end{lemma}
\begin{proof}
    Surjectivity on points is clear. Choose $(\cE,\phi)\in\Gamma_{D}^{>k}$ and $(\cE',\phi')\in \Gamma_{D_2}^{>k}$. Since an edge 
    $$(\cE,\phi|_{D_2})\to (\cE',\phi')$$
    in $\Gamma_{D_2}^{>k}$ is represented by an inclusion $f:\cE'\to\cE$ which is an isomorphism outside $D_2$, it lifts uniquely to an edge $(\cE,\phi)\to (\cE',\tilde\phi')$ in $\Gamma_{D}^{>k}$ for some extension $\tilde\phi'$ of $\phi'$ to $D$. This shows the second part in the definition of the covering map.
    
    For the third part, take $v\in\Gamma_{D}^{>k}$ and $w\in\Gamma_{D_2}^{>k}$, and set $p:=p_{D,D_2}$, $v':=p(v)$.
    We claim that the number of edges from  $p^{-1}(w)$ to $v$ is at most the number of edges from $w$ to $v'$. Let us show how this claim proves the third property of a covering map. By the second property just proven, the total number of edges from $p^{-1}(w)$ to $p^{-1}(v')$ equals $|p^{-1}(w)|$ times the number of edges from $w$ to $v'$. On the other hand, by our claim, the number of edges from $p^{-1}(w)$ to $p^{-1}(v')$ is at most $|p^{-1}(v')|$ times the number of edges from $w$ to $v'$, with equality if and only if the third property is satisfied. But equality must be satisfied since $|p^{-1}(v')|=|p^{-1}(w)|$. So, we are done as soon as we prove the claim.

    Suppose the contrary: that there are more edges from $p^{-1}(w)$ to $v$ than from $w$ to $v'$. Choose representatives in each isomorphism class of level structures for $w=(\cE,\phi_{D_2})$, $v=(\cE',\phi_{D_2}',\phi_{D_1}')$, and $(\cE,\phi_{D_2},\psi_{D_1})$ for any element in $p^{-1}(w)$. By Proposition \ref{p:Hecke_correspondence_for_vector_bundles}, edges from each $(\cE,\phi_{D_2},\psi_{D_1})$ to $(\tilde\cE,\tilde\phi_{D_2},\tilde\psi_{D_1})$ are given by pairs $(f,\tau)$, where $f:\tilde \cE\hookrightarrow\cE$ is an isomorphism over $X\setminus x$ and $\tau$ is one of the chosen set of representatives $\tau_i$ of $K\Delta K=\sqcup \tau_i K$ ($\tau$ is redundant if $x\notin\supp D$). By our assumption, there exist at least two edges from $p^{-1}(w)$ to $v$ given by the same pair $(f,\tau)$. Write those edges as
    \begin{align*}
        (\cE,\phi_{D_2},\psi_{D_1}^1)&\to (\tilde \cE,\tilde\phi_{D_2},\psi_{D_1}^1\circ f),\\
        (\cE,\phi_{D_2},\psi_{D_1}^2)&\to (\tilde \cE,\tilde\phi_{D_2},\psi_{D_1}^2\circ f),
    \end{align*}
    where the first two elements in the triples on the right-hand side are the same since we use the same $(f,\tau)$ for both connections. By assumption, the level structures on the right represent the same vertex $v$, hence there exists an isomorphism
    $$
    \alpha: (\tilde \cE,\tilde\phi_{D_2},\psi_{D_1}^1\circ f)\simeq (\tilde \cE,\tilde\phi_{D_2},\psi_{D_1}^2\circ f).
    $$
    
    Using decompositions
    \begin{align*}
        \cE&=\cL_1\oplus\ldots\oplus\cL_n\\
        \tilde\cE&=\cL_1(-p_1x)\oplus\ldots\oplus\cL_n(-p_nx)
    \end{align*}
    as in Proposition \ref{p:cusp_connected_to_cusp}, we can regard $f$ and $\alpha$ as upper-triangular matrices with diagonal entries in $T(k)$. Moreover, $\alpha$ is strictly upper-triangular if $D_2\ne 0$ (see Theorem \ref{t:gln_image_of_aut_group}). Then $\beta_{D_1}:=(f\alpha f^{-1})|_{D_1}$ lies in $U(\cO_{D_1})$ if $D_2\ne 0$ and in $T(k)U(\cO_{D_1})$ otherwise. By Theorem \ref{t:gln_image_of_aut_group}, it lifts to an isomorphism
    $$
    \beta: (\cE,\phi_{D_2},\psi_{D_1}^1)\to (\cE,\phi_{D_2},\psi_{D_1}^2),
    $$
    which implies that $\psi_{D_1}^1=\psi_{D_1}^2$ since we chose one representative from each isomorphism class. This is a contradiction, so we are done.
\end{proof}

\begin{theorem}\label{t:graph_for_ramification_at_cusp}
    Assume that $\supp D_2\ne \{x\}$. There is a decomposition of $\Gamma_{D}^{>\deg D}=\bigsqcup_{i=1}^N\Gamma_i$ into
    $$
    N=|T(\cO_{D_1})|/|T(\FF_q)|^{\delta_{D_2,0}}\cdot|\operatorname{Fl}_n(\cO_{D_1})|
    $$
    subgraphs such that
    \begin{enumerate}
        \item There are no edges between $\Gamma_i$ and $\Gamma_j$ for any $i\ne j$,
        \item The projection $p_{D,D_2}$ maps $\Gamma_i$ isomorphically onto $\Gamma_{D_2}^{>\deg D}$.
    \end{enumerate}
\end{theorem}

\begin{proof}
    Let $\Gamma_{D_2}'$ be a connected component of $\Gamma_{D_2}^{>\deg D}$. Choose $(\cE,\phi_{D_2})\in \Gamma_{D_2}'$ and let 
    $$
    p_{D,D_2}^{-1}((\cE,\phi_{D_2})):=\{(\cE^i,\phi_{D}^i)\}_{i=1..N}.
    $$
    
    Let $\Gamma_i'$ be the connected component of $\Gamma_{D}^{>\deg D}$ containing $(\cE^i,\phi_{D}^i)$. We claim that it satisfies the properties $(1)$ and $(2)$ with respect to $\Gamma_{D_2}'$, from which we will have the theorem proven after taking the union over all connected components $\Gamma_{D_2}'$.

    First, we prove $(1)$. Suppose that $\Gamma_i'=\Gamma_j'$ for some $i,j$. It means that there is a sequence of bundles
    $$
    \cE^i=\cE_1,\cE_2,\ldots,\cE_k=\cE^j
    $$
    with level structures at $D$ such that $\cE_s$ is connected to $\cE_{s+1}$ by an edge. This means that there is a map between $\cE_s$ and $\cE_{s+1}$ which is an isomorphism outside of $x$ and preserves the level structures at $D$. By Proposition \ref{p:cusp_connected_to_cusp}, we can write 
    \begin{equation}\label{eq:decomp_in_a_loop}
        \cE_s=\cL_1(p_1^sx)\oplus\ldots\oplus\cL_n(p_n^sx)
    \end{equation}
    for certain line bundles $\cL_s$ and integers $0\le p_t^s\le 1$. By the same Proposition, all the maps between $\cE_s$ and $\cE_{s+1}$ are upper-triangular with canonical embeddings at the diagonal in terms of the decomposition (\ref{eq:decomp_in_a_loop}). Write $D=D'+d_x[x]$ for $x\notin\supp D'$. When we restrict to $D'$, all these maps become strictly upper-triangular isomorphisms, and therefore compose to a strictly upper-triangular isomorphism $f:\cE^i|_{D'}\to\cE^j|_{D'}$. 

    Our assumptions on $D_2$ leave us with two cases: $D_2=0$ or $D_2$ contains a point $y\ne x$ in its support. Assume the first case. By Theorem \ref{t:gln_image_of_aut_group}, since $f|_{D}\in U(\cO_D)$, this restriction extends to an automorphism of $\cE_1=\cE$. It identifies $\phi^i$ with $\phi^j$, which proves the claim in this case.
    
    Assume now that $D_2$ contains a point $y\ne x$ in its support. Since $f$ is composed by Hecke modifications and it is an isomorphism over $D'$, the level structures $\phi^i|_{D'}$ and $\phi^j|_{D'}$ are isomorphic. Let $a$ and $b$ be automorphisms of $\cE_1=\cE_k$ such that $a$ maps $\phi^i|_{D_2}$ to $\phi^j|_{D_2}$, and $(b\circ f)|_{D'}$ maps $\phi^i|_{D'}$ to $\phi^j|_{D'}$. We want to show that the automorphism
    $$
    ((b\circ f)|_{D_1},a|_{D_2})
    $$
    of $\cE_1|_{D}$ extends to an automorphism of $\cE_1$, which will prove the claim because this automorphism will identify $\phi^i$ with $\phi^j$. Equivalently, we want to show that $(a^{-1}bf)|_{D_1}$ lifts to an automorphism of $\cE_1$ preserving $\phi^1|_{D_2}$. By Theorem \ref{t:gln_image_of_aut_group}, it is enough to show that $(a^{-1}bf)|_{D_1}\in U(\cO_{D_1})$.

    We already know that $(a^{-1}bf)|_{D'}\in T(k)\ltimes U(\cO_{D'})$. However, since both $bf$ and $a$ send $\phi^i|_{\supp D_2\cap \supp D'}$ to $\phi^j|_{\supp D_2\cap \supp D'}$, we have
    $$
    a^{-1}bf\phi^i|_{\supp D_2\cap \supp D'}=a^{-1}\phi^j|_{\supp D_2\cap \supp D'}=\phi^i|_{\supp D_2\cap \supp D'}.
    $$
     This shows that $a^{-1}bf|_{\supp D_2\cap \supp D'}=\mathrm{id}$, which in particular implies that $a^{-1}bf|_{D'}\in U(\cO_{D'})$ (recall that $T(k)$ is embedded diagonally into $T(k)\ltimes U(\cO_{D'})$). Further restriction to $D_1$ finishes the proof of (1).

    According to Lemma \ref{l:covering_map_GL}, the restriction of the map $p_{D,D_2}$ to $\Gamma_i'$ is a covering map to $\Gamma_{D_2}'$. By above, this map is a bijection on vertices. Therefore, it is an isomorphism, which finishes the proof.
\end{proof}

The above proof suggests that in the case $D_2=d_x[x]$, the monodromy of the projection map  $\Gamma_{D}^{>\deg D}\to \Gamma_{d_x[x]}^{>\deg D}$ lies in $T(k)$. The precise statement is the following.
\begin{theorem}\label{t:graph_for_ramification_at_cusp_over_x}
    Assume that $d_x\ge 1$. Take $(\cE,\phi^1), (\cE,\phi^2)\in \Gamma_{D}^{>\deg D}$ lying in the same connected component of $\Gamma_{D}^{>\deg D}$ such that $(\cE,\phi^1|_{d_x[x]})= (\cE,\phi^2|_{d_x[x]})$ in $\Gamma_{d_x[x]}^{>\deg D}$. Then there exists an automorphism $\alpha\in\Aut(\cE)$ such that $\alpha$ identifies $\phi^1|_{d_x[x]}$ with $\phi^2|_{d_x[x]}$ and $\alpha\phi^2|_{D'}=t\phi^1|_{D'}$ for some 
    $$
    t\in T(k)\simeq \bigoplus_i \Aut(\cL_i)\subset \Aut(\cE).
    $$
\end{theorem}

\begin{proof}
    Repeating the lines and the notation from the proof of Theorem \ref{t:graph_for_ramification_at_cusp} with $\cE^i=\cE^j=\cE$, we can write
    $$
    (a^{-1}bf)|_{D'}=ut,\qquad t\in T(k), u\in U(\cO_{D'}).
    $$
    Then $(a^{-1}bft^{-1})|_{D'}\in U(\cO_{D'})$, which by Theorem \ref{t:gln_image_of_aut_group} implies that the automorphism
    $$
    (t(b\circ f)^{-1}|_{D'},a^{-1}|_{d_x[x]})
    $$
    of $\cE|_{D}$ lifts to an automorphism $\alpha$ of $\cE$. It sends $\phi^2|_{d_x[x]}$ to $\phi^1|_{d_x[x]}$ and $\phi^2|_{D'}$ to $t\phi^1|_{D'}$.
\end{proof}

See Example \ref{ex:nontrivial_monodromy_when_ramified_at_x} showing that the monodromy can be non-trivial. 

\begin{remark}\label{r:GL_to_PGL}
    {\rm
        Using Proposition \ref{p:from_GL_to_PGL}, we can see that results in Theorem \ref{t:gln_image_of_aut_group}, Corollary \ref{cor:gln_preimage_size}, and Theorem \ref{t:graph_for_ramification_at_cusp} stay the same for $G=\PGL_n$ if we replace the torus $T$ for $\GL_n$ by the torus for $\PGL_n$. The final formula in Corollary \ref{cor:gln_preimage_size} will then be
        $$
        \# p_{D,D_2}^{-1}(\cE)=\frac{q^{(\deg D_1-\deg (D_1)_{\mathrm{red}})\frac{(n+2)(n-1)}2}}{(q-1)^{(n-1)\delta_{D_2,0}}}\prod_{y\in\operatorname{supp}D_1}\prod_{k=2}^n(q^{k\deg y}-1).
        $$
        For a more direct proof, see Corollary \ref{cor:deep_cusp_fiber_size_reductive}.
    }
\end{remark}

\subsection{Generalities on infinite graphs} \label{s:General graphs}
Let $\Gamma$ be a graph for which the ingoing and outgoing degrees of each vertex are finite. Let $A(\Gamma)$ be the adjacency matrix of $\Gamma$. To say that $v\in\CC\Gamma:=\CC^\Gamma$ is an eigenvector of $A(\Gamma)$ with eigenvalue $\lambda$ is equivalent to saying that
$$
\lambda v_i=\sum_{i\to j}v_j,
$$
where edges $i\to j$ are counted with multiplicities. Note that this is the exact same problem we were solving for graphs of Hecke operators, so we can think of the Hecke operator as the adjacency matrix of its graph $\Gamma$ acting on $\CC\Gamma$. 

\begin{definition}
    We call a disjoint decomposition
    \begin{equation}\label{eq:graph_decomp}
        \Gamma=\Gamma'\sqcup\Gamma_1\sqcup\Gamma_2\sqcup\ldots 
    \end{equation}
     into finite subsets {\bfseries propagative} if the linear adjacency maps $\CC\Gamma_i\to\CC\Gamma_{i-1}$ induced by the arrows from vertices of $\Gamma_{i-1}$ to the vertices of $\Gamma_i$ are surjective for every $i>1$.

    We call such a decomposition {\bfseries strictly propagative} if the above mentioned maps are isomorphisms.

    For $\lambda\in\CC$, we denote by $\dim_\lambda \Gamma$ the dimension of the $\lambda$-eigenspace of the adjacency matrix of $\Gamma$.
\end{definition}
\begin{theorem}\label{t:eigenforms_for_graphs_with_decomposition}
    Choose $\lambda\in \CC$, and let a graph $\Gamma$ admit a propagative decomposition (\ref{eq:graph_decomp}). Let $A_i:\CC\Gamma_i\to\CC\Gamma_{i-1}$ be the linear adjacency maps induced by the arrows from vertices of $\Gamma_{i-1}$ to the vertices of $\Gamma_i$. Then 
    
    $$
    \sup_{i\ge 1}|\Gamma_i|\le\dim_\lambda\Gamma\le\dim_\lambda\Gamma'+\sup_{i\ge 1}|\Gamma_i|.
    $$
    
    In particular, if the decomposition is strictly propagative, then
    $$
    |\Gamma_1|\le\dim_\lambda\Gamma\le\dim_\lambda\Gamma'+|\Gamma_1|.
    $$
\end{theorem}
\begin{proof}
    Complete the picture by the following adjacency linear maps:
    \begin{align*}
        B_i&:\CC\Gamma_{i-1}\to\CC\Gamma_{i},\\
        M_i&:\CC\Gamma_{i}\to\CC\Gamma_{i},\\
        B&:\CC\Gamma'\to\CC\Gamma_{1},\\
        A&:\CC\Gamma_{1}\to\CC\Gamma',\\
        M&:\CC\Gamma'\to\CC\Gamma'.
    \end{align*}
    The whole adjacency matrix of $\Gamma$ is decomposed into blocks with these matrices.

    Let us try to construct a $\lambda$-eigenvector $v$ of $\Gamma$. Write $v=(v',v_i)$, where $v_i\in\CC\Gamma_i$ and $v'\in\Gamma'$. The eigenvalue condition translates to:
    \begin{itemize}
        \item $(\lambda-M)v'=Av_1$,
        \item $\lambda v_1=M_1v_1+Bv'+A_2 v_{2}$,
        \item $\lambda v_i=M_iv_i+B_iv_{i-1}+A_{i+1} v_{i+1}$.
    \end{itemize}
    
    Let $p$ be the dimension of the space of solutions of the first equation in $v'$ and $v_1$. Note that with fixed $v'$ and $v_1$, the space of solutions of the second equation has dimension $\dim \ker A_2$. Similarly, with fixed $v_i$ and $v_{i-1}$, the space of solutions of the third equation as dimension $\dim \ker A_{i}$. Therefore,
    $$
    \dim_\lambda\Gamma=p+\sum_{i\ge 2}\dim\ker A_i.
    $$
    Since $A_i$ are epimorphisms, we have $\dim\ker A_i=|\Gamma_i|-|\Gamma_{i-1}|$. Therefore,
    $$
    \dim_\lambda\Gamma=p-|\Gamma_1|+\sup_{i\ge 1}|\Gamma_i|.
    $$
    
    Now, let us give estimates on $p$. By definition,
    $$
    p=\dim(\im(\lambda-M)\cap\im A)+\dim\ker(\lambda-M)+\dim\ker A.
    $$
    By standard linear algebra,
    $$
    \dim\im A-\dim\ker(\lambda-M)\le\dim(\im(\lambda-M)\cap\im A)\le\dim\im A.
    $$
    Since $\dim\im A+\dim\ker A=|\Gamma_1|$, this gives
    $$
    |\Gamma_1|\le p\le |\Gamma_1|+\dim\ker(\lambda-M).
    $$
    Plugging this in the above expression for $\dim_\lambda\Gamma$ finishes the proof.
\end{proof}

\begin{remark}
    {\rm
        Note that the position of $\dim_\lambda\Gamma$ between the bounds in Theorem \ref{t:eigenforms_for_graphs_with_decomposition} depends only on the intersection of $\im(M-\lambda)$ and $\im A$. Indeed, assume that the dimensions of these spaces are known. If their intersection has the minimal possible dimension, then $\dim_\lambda\Gamma$ equals the lower bound. If the intersection dimension has the maximal possible dimension, then the upper bound is attained. Because the minimal intersection represents the generic case for arbitrary two linear maps with known ranks, it is reasonable to conjecture that the lower bound is always attained for graphs of Hecke operators. All computations in our paper and \cite{AB24} support this conjecture; however, we would love to get more computational evidence before putting faith in it.
    }
\end{remark}
\begin{corollary}\label{cor:generic_eigenspace_for_graphs}
    Keep the assumptions of Theorem \ref{t:eigenforms_for_graphs_with_decomposition}. Then, for all but finitely many $\lambda$,
    $$
    \dim_\lambda\Gamma=\sup_{i\ge 1}|\Gamma_i|.
    $$
\end{corollary}
\begin{proof}
    Since $|\Gamma'|$ is finite, its adjacency matrix has finitely many eigenvalues. We have $\dim_\lambda \Gamma'=0$ for any $\lambda$ not equal to any of them, in which case Theorem \ref{t:eigenforms_for_graphs_with_decomposition} gives the desired equality.
\end{proof}

We note one beautiful corollary of the proof of Theorem \ref{t:eigenforms_for_graphs_with_decomposition}. 
    Assume that a graph $\Gamma$ admits a strictly propagative decomposition (\ref{eq:graph_decomp}) with $|\Gamma_1|<\infty$ and $|\Gamma'|<\infty$ (this condition on $\Gamma'$ can be weakened: see Remark \ref{r:vector_bundle_of_eigenforms_in_ramification}).
    Let $M$ be the adjacency matrix of $\Gamma'$, and define
    \begin{align*}
    R&:=\{\lambda\in\CC^\times:\det(\lambda-M)\ne 0\},\\
    Q&:=\{(\lambda,f)\in R\times\CC[\Gamma]:\Phi(f)=\lambda f\}.
    \end{align*}
We have an obvious projection map $p:Q\to R$.
\begin{corollary}\label{cor:vector_bundle_of_eigenforms}
    The map $p:Q\to R$ turns $Q$ into a trivial algebraic vector bundle over $R$ of rank $|\Gamma_1|$, in the sense that the projections
    \begin{align*}
        Q&\to \CC,\\
        (\lambda,f)&\mapsto f(x)
    \end{align*}
    for all $x\in\Gamma$ are algebraic.
 \end{corollary}
\begin{proof}
    Using the notation of the proof of Theorem \ref{t:eigenforms_for_graphs_with_decomposition}, we define a map
    \begin{align*}
        Q&\to R\times \CC^{|\Gamma_1|},\\
        (\lambda,v)&\mapsto (\lambda,v_1).
    \end{align*}
    This map is clearly an isomorphism. To prove that this isomorphism is algebraic, we need to show that entries of $v'$ and $v_{i}$ for $i\ge 2$ depend algebraically on $v_1$ and $
    \lambda$. Since entries of $v_i$ for $i\ge 2$ depend linearly on $v'$, $v_1$, and $\lambda$, it is enough to prove that $v'$ depends algebraically on $v_1$. We have
    $$
    v'=(M-\lambda)^{-1}Av_1.
    $$
    
    By definition of $R$, the entries of the (finite) matrix $(M-\lambda)^{-1}$ are elements in the coordinate ring of $R$. This makes $v'$ depend algebraically on $v_1$ and $\lambda$, as desired.
\end{proof}

\subsection{Spaces of eigenforms for $\PGL_2$ unramified at $x$}\label{ss:unramified_PGL2}

Fix a point $x\in |X|$ of degree $r$ and a divisor $D$ on $X$ such that $x\notin\operatorname{supp}D$. Let $\Gamma$ be the graph of the corresponding Hecke operator. Use the following decomposition:
\begin{itemize}
    \item $\Gamma_i:=\{\cL_1\oplus\cL_2:ri<\deg \cL_1-\deg\cL_2-(2g-2+\deg D)\le r(i+1)\}/\sim$,
    \item $\Gamma':=\Gamma\setminus\bigsqcup_{i\ge 1}\Gamma_i$,
\end{itemize}
where $\sim$ is the equivalence relation given by tensoring with a line bundle with a level structure at $D$. In other words, $\Gamma_i$ cover all the bundles deep in the cusp of $\Bun_G$. By Proposition \ref{p:cusp_connected_to_cusp} and Theorem \ref{t:eigenforms_for_graphs_with_decomposition}, the edges from $\Gamma_i$ to $\Gamma_{i+1}$ are of the form
$$
(\cL_1\oplus\cL_2)\to(\cL_1\oplus\cL_2(-x))
$$
such that the level structures are compatible with the canonical embedding of the right term into the left term. This shows that the arrows between $\Gamma_i$ and $\Gamma_{i+1}$ form a bijection between these sets. With the notations from the previous section, the graph is presented on Figure \ref{fig:PGL2_unramified}.

\begin{figure}
    \centering
\scalebox{.8}{
\includegraphics{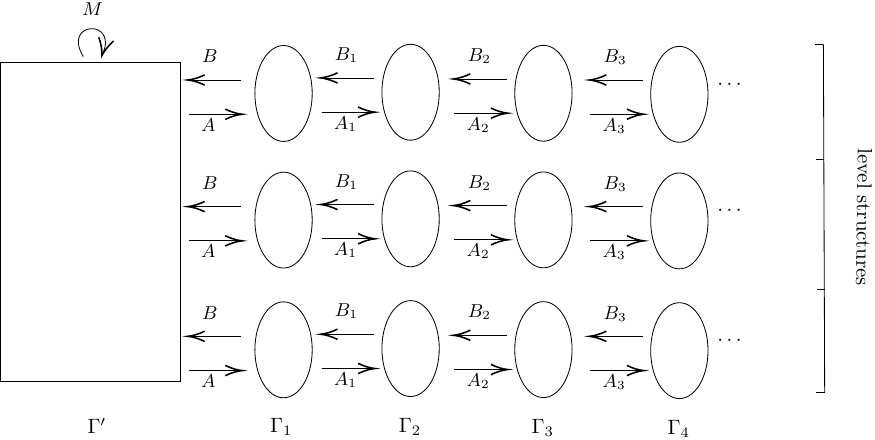}
}
\caption{Graph for $\PGL_2$ unramified at $x$.}
\label{fig:PGL2_unramified}
\end{figure}

Therefore, the adjacency linear map $\CC\Gamma_{i+1}\to\CC\Gamma_i$ is an isomorphism. We have
$$
|\Gamma_i|=r\cdot|\operatorname{Pic}^0(X)(\FF_q)|,
$$
hence by Theorem \ref{t:eigenforms_for_graphs_with_decomposition} and Remark \ref{r:GL_to_PGL}, we get

\begin{align*}
r\cdot|\operatorname{Pic}^0(X)(\FF_q)|\cdot |T(\cO_D)/T(\FF_q)|&\cdot|\operatorname{Fl}_2(\cO_D)|\le\dim_\lambda\Gamma\le \\
&\le\dim_\lambda\Gamma'+r\cdot|\operatorname{Pic}^0(X)(\FF_q)|\cdot |T(\cO_D)/T(\FF_q)|\cdot|\operatorname{Fl}_2(\cO_D)|
\end{align*}
if $D\ne 0$, and
\begin{align*}
r\cdot|\operatorname{Pic}^0(X)(\FF_q)|\le\dim_\lambda\Gamma\le \dim_\lambda\Gamma'+r\cdot|\operatorname{Pic}^0(X)(\FF_q)|
\end{align*}
if $D=0$. Substituting 
\[
|T(\cO_D)|/|T(\mathbb F_q)|\cdot |\operatorname{Fl}_2(\cO_D)|
=
\frac{q^{2(\deg D-\deg D_{\mathrm{red}})}}{q-1}\prod_{y\in\operatorname{supp}D} (q^{2\deg y}-1),
\]
we finally obtain
\begin{theorem} \label{t:PGL2_ramified_not_at_x}
    Let $x$ be a point of degree $r$, and let $\Phi_x$ be the corresponding Hecke operator. Consider ramification at divisor $D=\sum_y d_y[y]$, whose support doesn't contain $x$. If $D\ne 0$, then
    \begin{align*}
r\cdot|\operatorname{Pic}^0(X)(\FF_q)|&\cdot\frac{q^{2(\deg D-\deg D_{\mathrm{red}})}}{q-1}\prod_{y\in\operatorname{supp}D} (q^{2\deg y}-1)\le\dim_\lambda\Gamma\le\\
&\le\dim_\lambda\Gamma'+r\cdot|\operatorname{Pic}^0(X)(\FF_q)|\cdot \frac{q^{2(\deg D-\deg D_{\mathrm{red}})}}{q-1}\prod_{y\in\operatorname{supp}D}(q^{2\deg y}-1),
\end{align*}
where $D_{\mathrm{red}}=\sum_{y\in\operatorname{supp}D}1\cdot[y]$. If $D=0$, then
\begin{align*}
r\cdot|\operatorname{Pic}^0(X)(\FF_q)|\le\dim_\lambda\Gamma\le \dim_\lambda\Gamma'+r\cdot|\operatorname{Pic}^0(X)(\FF_q)|.
\end{align*}
These turn into equalities for all but finitely many $\lambda$.
\end{theorem} 
\begin{remark}\label{r:bounded_growth}
    {\rm
        Since we are interested in automorphic forms, we need to check the moderate growth condition. Since the matrices $A_i$ are the same matrix, the growth of each eigenform over $\Gamma_n$ is asymptotically bounded by $(\lambda \lambda')^{n\deg x}$, where $\lambda'$ is the highest eigenvalue of $A_i$. Since the valuation of bundles in $\Gamma_n$ is $O(q^{n\deg x})$, the moderate growth condition is satisfied. See \cite{AB24} for details on this condition.
    }
\end{remark}

\subsection{Spaces of eigenforms for $\PGL_2$ ramified at $x$} \label{ss:ramified_PGL2}
We study the graph of the Hecke operator $\Phi_x$ with ramification at $d[x]$. Let us pick a bundle $\cL_1\oplus\cL_2$ in the cusp with 
$$
\deg \cL_1-\deg\cL_2>2g-2+d\cdot\deg x.
$$
We will write a level structure on this bundle as a matrix
$$
\begin{pmatrix}
    a_{11}&a_{12}\\
    a_{21}&a_{22}
\end{pmatrix},
$$
where $a_{ij}:\cL_i|_{\cO_{d[x]}}\to \cO_{d[x]}e_j$, where $e_1,e_2$ is the standard basis of $\cO_{d[x]}^2$. We study what bundles can be connected to this bundle with level structure $\tilde a=(a_{ij})$. Lift this matrix to an invertible matrix $(\tilde a_{ij})$ over $\cO_x$. We fix trivializations $l_1$ and $l_2$ of $\cL_1$ and $\cL_2$ over $\cO_x$, so that we can treat $(\tilde a_{ij})$ as a matrix with entries in $\cO_{x}$. According to Theorem \ref{t:ramified_Hecke_correspondences}, there is exactly one bundle connected to this level structure, identified by $\tilde a^{-1}\Delta\cO_x^{\oplus 2}\subset \cL_1\oplus\cL_2$. Explicitly,
$$
\tilde a^{-1}\Delta\cO_x^{\oplus 2}=\{(p,q)\in \cL_1\oplus\cL_2:\tilde a_{11} p+\tilde a_{12}q\in\pi_x\cO_x\}.
$$
From this we see that if $a_{11}\in\pi_x\cO_{d[x]}$, then the bundle is $\cL_1\oplus\cL_2(-x)$, and if $a_{11}\in\cO_{d[x]}^\times$, then the bundle is $\cL_1(-x)\oplus\cL_2$. Let us study these two cases:

{\bfseries Case I: $a_{11}\in\pi_x\cO_{d[x]}$.} The matrix $(a_{ij})$ is defined up to an automorphism of $\cL_1\oplus\cL_2$, which by the proof of Theorem \ref{t:gln_image_of_aut_group} acts by right multiplication by an upper-triangular matrix from the group
$$
T(k)\ltimes U(\cO_{d[x]})=\begin{pmatrix}
    k^\times&\cO_{d[x]}\\
    0&k^\times
\end{pmatrix}.
$$
Using this group and the center of $\GL_2(\cO_{d[x]})$ (the group we are having all matrices in is $\PGL_2$), we reduce the matrix $(a_{ij})$ to the form
\begin{equation}\label{eq:matrix_form_at_infinity}
    \begin{pmatrix}
    \pi_x\cO_{d[x]}&1\\
    \cO_{d[x]}^\times/k^\times&0
\end{pmatrix},
\end{equation}
which is precisely the fiber of $\infty\in\PP^1(k_x)$ under the projection
\begin{equation}\label{eq:projection_to_P1_for_ramified_case}
    G(\cO_{d[x]})/(T(k)\ltimes U(\cO_{d[x]}))\to G(k_x)/B(k_x)\simeq \PP^1(k_x).
\end{equation}
Note that this reduction corresponds to a different decomposition of $\cL_1\oplus\cL_2$ in the direct sum of $\cL_1$ and $\cL_2$, which also changes the chosen trivialization of $\cL_1\oplus\cL_2$. From now, we fix a decomposition $\cL_1\oplus\cL_2$ and trivializations $l_1,l_2$ of $\cL_1,\cL_2$ such that the matrix $(a_{ij})$ has the form \ref{eq:matrix_form_at_infinity}. Moreover, fix a trivialization of $\cL_1\oplus\cL_2(-x)$ such that the inclusion matrix is given by
$$
\sigma:=\mat100{\pi_x}.
$$

Write $a_{11}=\pi_x^sp$ for some $s\ge 1$ and $p\in k_x^\times$. Let the level structure on $\cL_1\oplus\cL_2(-x)$ be given by a matrix $(b_{ij})$ written in terms of the above trivialization (equivalently, in the basis $l_1$ and $\pi_x l_2$). By Theorem \ref{t:ramified_Hecke_correspondences}, any such level structure $(b_{ij})$ is given by 
\begin{align}
    (b_{ij})&=\Delta^{-1}\begin{pmatrix}
    1&\pi_x^{d}C\\
    0&1
\end{pmatrix}\tilde a \sigma\nonumber\\
&=\begin{pmatrix}
    \pi_x^{-1}&\pi_x^{d-1}C\\
    0&1
\end{pmatrix}\begin{pmatrix}
    \pi_x^{s}p&\pi_x\\
    a_{12}&0
\end{pmatrix}\nonumber\\
&=\begin{pmatrix}
    \pi_x^{s-1}p+a_{12}\pi_x^{d-1}C&1\\
    a_{12}&0
\end{pmatrix}.\label{eq:ramified_connections_of_infinity}
\end{align}
 Observe that the matrix $(b_{ij})$ must satisfy $b_{21}\in\cO_{d[x]}^\times$. Therefore, it lies in the preimage of $\PP^1(k_x)\setminus 0$ under the projection (\ref{eq:projection_to_P1_for_ramified_case}). Splitting the preimage of $\infty$  under this projection into regions corresponding to the $\pi_x$-valuation of $\tilde a_{11}$, we get Figure \ref{fig:PGL2_ramified_partial}, where we are yet to add arrows which go from right to left.

\begin{figure}
    \centering
\scalebox{.9}{
\includegraphics{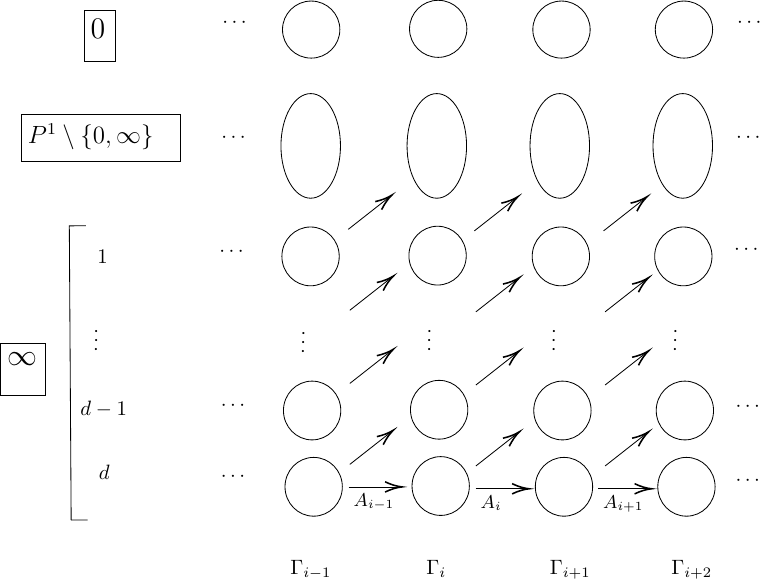}
}
\caption{$\infty$-part of the $\PGL_2$-graph ramified at $d[x]$.}
\label{fig:PGL2_ramified_partial}
\end{figure}

The matrices $A_i$ on this figure correspond to $s=d$, $p=0$, $C=0$, in which case
$$
(a_{ij})=
\begin{pmatrix}
    0&a_{12}\\
    a_{21}&0
\end{pmatrix},\quad(b_{ij})=
\begin{pmatrix}
    0&a_{12}\\
    a_{21}&0
\end{pmatrix}.
$$
This shows that $A_i$ are bijections on vertices of this locus, and hence the corresponding matrices are invertible. We also remark that all matrices (\ref{eq:ramified_connections_of_infinity}) are of the form (\ref{eq:matrix_form_at_infinity}), and hence give pairwise distinct level structures. Therefore, all edges have multiplicity $1$ in this case. 

{\bfseries Case II:  $a_{11}\in\cO_{d[x]}^\times$.} Like in the previous case, we can reduce the matrix $(a_{ij})$ to the form
\begin{equation}\label{eq:matrix_form_away_from_infinity}
(a_{ij})\in
    \begin{pmatrix}
    1&0\\
    \cO_{d[x]}&\cO_{d[x]}^\times/k^\times
\end{pmatrix}.
\end{equation}
In this case, the bundle is $\cL_1(-x)\oplus\cL_2$, so we use its trivialization such that the inclusion matrix is
$$
\tau:=\mat{\pi_x}001.
$$
Then
\begin{align}
    (b_{ij})&=\Delta^{-1}\begin{pmatrix}
    1&\pi_x^{d}C\\
    0&1
\end{pmatrix}\tilde a \tau\nonumber\\
&=\begin{pmatrix}
    \pi_x^{-1}&\pi_x^{d-1}C\\
    0&1
\end{pmatrix}\begin{pmatrix}
    \pi_x&0\\
    \pi_xa_{21}&a_{22}
\end{pmatrix}\nonumber\\
&\equiv\begin{pmatrix}
    1&\pi_x^{d-1}a_{22}C\\
    \pi_xa_{21}&a_{22}
\end{pmatrix}\sim \begin{pmatrix}
    1&0\\
    \pi_xa_{21}&a_{22}
\end{pmatrix},\label{eq:ramified_connections_away_from_infinity}
\end{align}
where we wrote equivalence modulo $\pi_x^d$ in the last line and brought the matrix to the form (\ref{eq:matrix_form_away_from_infinity}) in the last step using right multiplication by $\begin{pmatrix}
    1&-\pi_x^{d-1}C\\
    0&1
\end{pmatrix}$. Note that we get the same matrix for any $C$, so every edge in this case has multiplicity $|k_x|$, and hence every vertex is connected to exactly one other vertex. Split the preimage of $0\in\PP^1(k_x)$ under the projection map (\ref{eq:projection_to_P1_for_ramified_case}) into regions corresponding to the  $\pi_x$-valuation of $a_{21}$, we finally get the graph on Figure \ref{fig:PGL2_ramified_full} ($q^r:=|k_x|$).

\begin{figure}
\centering
\scalebox{.9}{
\includegraphics{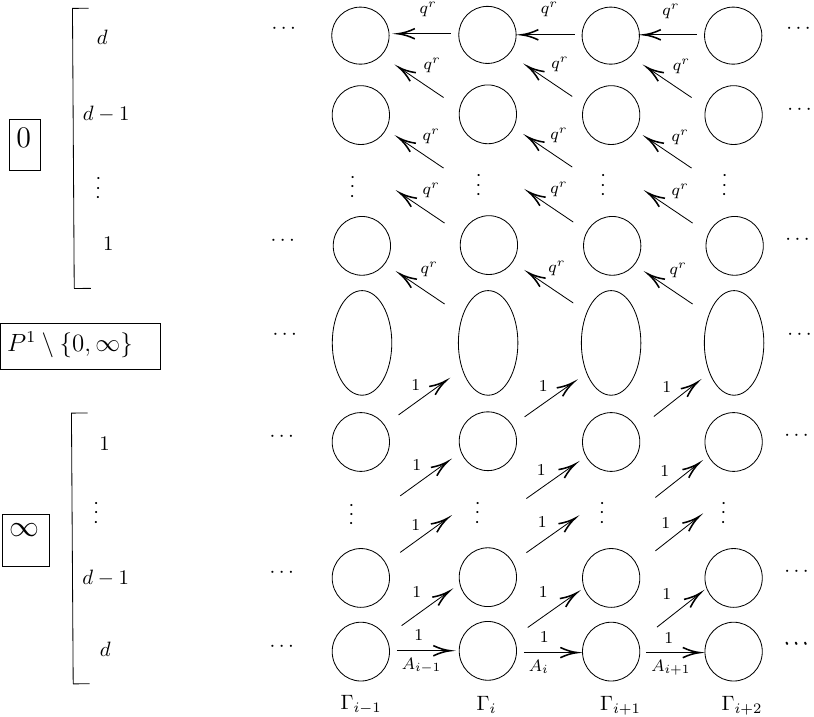}
}
\caption{$\PGL_2$-graph ramified at $d[x]$, where $|k_x|=q^r$.}
\label{fig:PGL2_ramified_full}
\end{figure}

To analyze eigenforms, define $\Gamma'$ as in Figure \ref{fig:PGL2_ramified_computation}.

\begin{figure}
\centering
\scalebox{.9}{
\includegraphics{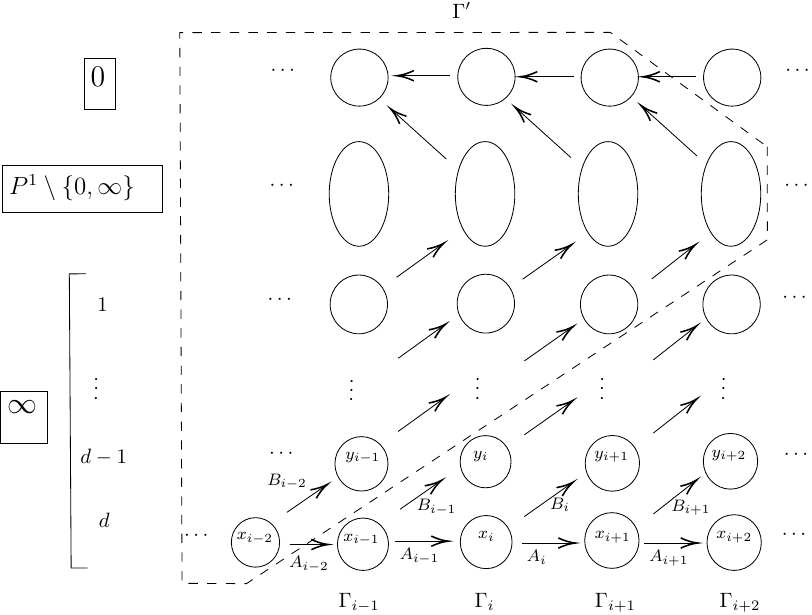}
}

\caption{Computation of eigenspaces for $\PGL_2$ ramified at $x$.}
\label{fig:PGL2_ramified_computation}
\end{figure}

Note that as soon as we have $x\in\CC\Gamma'$, there is a unique extension of it to $\tilde x\in \CC[\Gamma\setminus\Gamma^{\infty, d}]$  satisfying the $\lambda$-eigenvalue condition on $\Gamma\setminus(\Gamma'\cup\Gamma^{\infty, d})$. Indeed, using the rule
$$
\lambda\cdot\text{vertex}=\sum\text{its neighbors}
$$
and the fact that every vertex has at least one neighbor, we first uniquely extend $x$ to the whole $0$ level, then to the $\PP^1\setminus\{0,\infty\}$ level, and finally to the $1,2,\ldots,d-1$ levels at the $\infty$ level.

We want to study extensions of $\tilde x$ to the whole graph such that it satisfies the $\lambda$-eigenvalue condition at any vertex. Let $M$ be the adjacency matrix of $\Gamma'$. Then the desired extensions must satisfy
\begin{align*}
    (\lambda-M)x+A_{i-2}x_{i-1}&=0,\\
    A_{j}x_{j+1}+B_{j}y_{j+1}&=\lambda x_j.
\end{align*}

Let $p$ be the dimension of the space of solutions of the first equation. The next equation says
$$
\lambda x_{i-1}-A_{i-1} x_{i}=B_{i-1}y_i.
$$
Since $y_i$ is known and $A_{i-1}$ is invertible, we get a unique solution $x_i$.  Continuing in this manner, we build all $x_j$ uniquely. Therefore,
$$
\dim_\lambda\Gamma=p.
$$
Using the same arguments as in the proof of Theorem \ref{t:eigenforms_for_graphs_with_decomposition}, we get
$$
\dim\im A_{i-2}\le p\le\dim\im A_{i-2}+\dim\ker(\lambda-M)=\dim\im A_{i-2}+\dim_\lambda\Gamma'.
$$
Since $A_{i-2}$ is invertible, $\dim\im A_{i-2}$ equals the number of vertices in the locus labelled by the variable $x_{i-1}$. There are $\deg x\cdot|\operatorname{Pic}^0(X)(k)|$ bundles there, and each of them has 
$$
|\cO_{d[x]}^\times/k^\times|=\frac{(|k_x|-1)|k_x|^{d-1}}{|k|-1}
$$
level structures (one for each matrix of the form (\ref{eq:matrix_form_at_infinity}) with zero at the upper-left corner). Therefore, the final statement is the following:
\begin{theorem} \label{t:PGL2_ramified_at_x}
    Let $x$ be a point of degree $r$, $k=\FF_q$, and let the Hecke operator $\Phi_x$ at $x$ be ramified at $d[x]$. Let $\lambda\ne 0$. Then
    $$
        r\cdot|\operatorname{Pic}^0(X)(\FF_q)|\cdot\frac{(q^r-1)q^{r(d-1)}}{q-1}\le\dim_\lambda\Gamma\le \dim_\lambda\Gamma'+r\cdot|\operatorname{Pic}^0(X)(\FF_q)|\cdot\frac{(q^r-1)q^{r(d-1)}}{q-1}.
    $$
        
    These are equalities for all but finitely many $\lambda$.
\end{theorem}
Using a similar reasoning as in Remark \ref{r:bounded_growth}, we see that each eigenform has bounded growth in this case, and therefore is an automorphic form.

Finally, notice that for ramification at $D=D'+d_x[x]$, the same arguments as above work and give the same structure of the graph as on Figure \ref{fig:PGL2_ramified_full}. The only additional argument we need is to show that the edges of multiplicity $q^r$ stay of multiplicity $q^r$ when additional ramification outside $x$ is added. In fact, in (\ref{eq:ramified_connections_away_from_infinity}), the last identification uses a strictly upper-triangular matrix, hence can be lifted to an automorphism of the whole bundle preserving the level structure outside $x$. This gives identification of the target level structures for all $C\in k_x$, so we are done.

Therefore, to get the most general result, we only need to compute the number of vertices in $x_i$ (see Figure \ref{fig:PGL2_ramified_computation}). We obtain

\begin{theorem} \label{t:PGL2_ramified_at_x_and_D}
    Let $x$ be a point of degree $r$, $k=\FF_q$, and let the Hecke operator $\Phi_x$ at $x$ be ramified at $D$. Assume that $d_x\ge 1$ and $D'\ne 0$. Let $\lambda\ne 0$. Then
    \begin{align*}
        r&\cdot|\operatorname{Pic}^0(X)(\FF_q)|\cdot\frac{(q^{r}-1)q^{2(\deg D'-\deg D'_{\mathrm{red}})+r(d_x-1)}}{q-1} \prod_{y\in\operatorname{supp}D'}(q^{2\deg y}-1)\\
        &\le\dim_\lambda\Gamma\\
        &\le \dim_\lambda\Gamma'+r\cdot|\operatorname{Pic}^0(X)(\FF_q)|\cdot\frac{(q^{r}-1)q^{2(\deg D'-\deg D'_{\mathrm{red}})+r(d_x-1)}}{q-1} \prod_{y\in\operatorname{supp}D'} (q^{2\deg y}-1),
    \end{align*}
    where $\Gamma'$ is some finite subgraph of $\Gamma$ defined using the methods above. These are equalities for all but finitely many $\lambda$.
\end{theorem}
\begin{proof}
    By Theorem \ref{t:graph_for_ramification_at_cusp}, adding ramification at $D'$ creates $|G(\cO_{D'})/U(\cO_{D'})|$ copies of the graph on Figure \ref{fig:PGL2_ramified_full} in the cusp. This multiplies the dimensions by $|G(\cO_{D'})/U(\cO_{D'})|$ (computed in Remark \ref{r:GL_to_PGL}), giving the stated formula.
\end{proof}

We note that the above results generalize and prove \cite[Conjecture 5.17]{AB24} for all but finitely many $\lambda$.

\begin{remark}\label{r:finite_set_of_lambda}
    {\rm
        We remark on the finite set of $\lambda$, for which the upper and lower bounds are distinct. These are precisely the eigenvectors of $\Gamma'$ defined in the above computations. Here are several properties of these eigenvectors giving a finite set without precise knowledge of $\Gamma'$:
        \begin{itemize}
            \item For fixed $D'$ and $d_x$, $\Gamma'$ is a graph on a finite set of vertices not in the $D$-cusp. Moreover, the outgoing degree of each vertex is $q^{\deg x}+1$ if $d_x=0$ (Proposition \ref{p:cusp_connections_degrees}) and $q^{\deg x}$ if $d_x>0$ (Theorem \ref{t:ramified_Hecke_correspondences}). The number of such graphs is clearly finite, which produces a finite number of possible eigenvalues.
            \item By Gershgorin circle theorem, eigenvalues of $\Gamma'$ are algebraic integers $\lambda$ of degree at most $|\Gamma'|$ with $|\lambda|\le q^{\deg x}+1$ if $d_x=0$ and $|\lambda|\le q^{\deg x}$ if $d_x>0$.
        \end{itemize}
    }
\end{remark}    

\begin{remark}\label{r:ramified_case_is_stictly_propagative}
    {\rm
        In the above computation, we could have used the methods from Section \ref{s:General graphs}. For this, instead of the subgraph $\Gamma'$ that we used in Figure \ref{fig:PGL2_ramified_computation}, we use the subgraph $\Gamma''$, which consists of all vertices excluding $x_{i-1},x_i,\ldots$. Then the decomposition 
        $$
        \Gamma=\Gamma''\sqcup \Gamma_{i-1}\sqcup\Gamma_{i}\sqcup\ldots
        $$
        is strictly propagative, and we can apply Theorem \ref{t:eigenforms_for_graphs_with_decomposition}. 
        
        As noted in the above computation, for any $\lambda\ne 0$, the values on $\Gamma'$ and the eigenform condition identify values on $\Gamma''$ uniquely. This implies that 
        $$
        \dim_{\lambda}\Gamma''=\dim_\lambda\Gamma',\qquad \text{for all }\lambda\ne 0,
        $$
        so we get Theorems \ref{t:PGL2_ramified_at_x} and \ref{t:PGL2_ramified_at_x_and_D} immediately.
    }
\end{remark}

\begin{remark}\label{r:vector_bundle_of_eigenforms_in_ramification}
    {\rm
        Although the graph $\Gamma''$ is infinite, we see that the values of any $\lambda$-eigenform on $\Gamma''\setminus \Gamma'$ depend linearly on the values on $\Gamma'$ and $\lambda$. In this case, the proof of Corollary \ref{cor:vector_bundle_of_eigenforms} still applies, and we get that in the case of any ramification, $\lambda$-eigenforms form a trivial vector bundle over $\CC^\times$ without eigenvalues of $\Gamma'$. 
    }
\end{remark}

\begin{example}
{\rm
    Let $X=\PP^1$, $\deg x=1$, $d=1$. In this example, we correct \cite[Figure 14]{Lor13} and \cite[Figure 5]{AB24}. For simplicity, we denote the vertex corresponding to $\cO(n)\oplus\cO$ with level structure $\phi\in\PP^1$ as $c_{n,\phi}$, and $c_0=\cO\oplus\cO$.
    
    We have already computed the edges from all points except $c_{1,\phi}$ for $\phi\ne \infty$ and $c_0$.

    Since every vertex has degree $q$ and we know that $c_{1,\phi}$ for $\phi\ne \infty$ are not connected to the bundle $\cO(2)\oplus\cO$, they are all connected to $c_0$ with multiplicities $q$.

    It remains to compute the outgoing edges from $c_0$. 
    Since any two level structures on $\cO\oplus\cO$ are equivalent, we may assume that the chosen level structure is given by 
    $$
    (a_{ij})=\begin{pmatrix}
        1&0\\
        0&1
    \end{pmatrix}.
    $$
    Since we want to represent bundles in the form $\cL_1\oplus\cL_2$ with $\deg \cL_1\ge\deg\cL_2$, we need to parametrize neighbors of $\cO\oplus\cO$ given by inclusions $\cO\oplus\cO(-1)\to \cO\oplus\cO$. Such inclusions are parametrized by matrices
    $$
    \sigma=\mat{1}{0}c{\pi_x}\text{ or }\mat{0}{\pi_x}{1}{0},\qquad c\in k.
    $$
    Since only the last matrix gives the lattice $\Delta(\cO_x\oplus\cO_x)=\pi_x\cO_x\oplus\cO_x$, it gives all the edges. We get that the level structures on $\cO(1)\oplus\cO$ connected to $(a_{ij})$ are
    \begin{align*}
    (b_{ij})&=\Delta^{-1}\begin{pmatrix}
    1&\pi_xC\\
    0&1
\end{pmatrix}\tilde a \sigma\\
&=\begin{pmatrix}
    \pi_x^{-1}&C\\
    0&1
\end{pmatrix}\begin{pmatrix}
    0&\pi_x\\
    1&0
\end{pmatrix}\nonumber\\
&=\begin{pmatrix}
    C&1\\
    1&0
\end{pmatrix},\qquad C\in k.
\end{align*}

Since $b_{21}\ne 0$, this gives all level structures but $0$. Summing everything up, we get the graph on Figure \ref{fig:PGL2_P1_ramified_example},
\begin{figure}
\centering
\scalebox{.9}{
\includegraphics{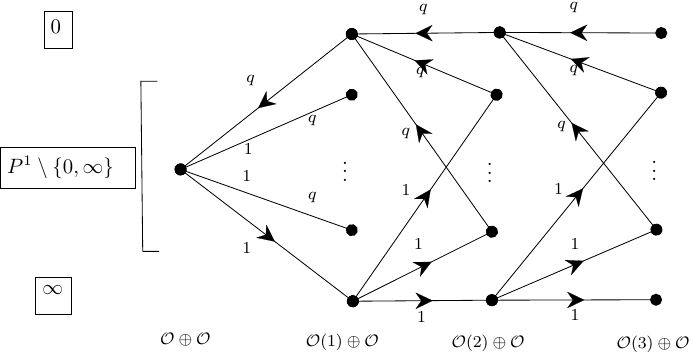}
}
\caption{Ramified graph for $G=\PGL_2$ and $X=\PP^1$ at $1\cdot[x]$, $\deg x=1$.}
\label{fig:PGL2_P1_ramified_example}
\end{figure}
    where the numbers mean edge multiplicities. 

    Let $f$ be a $\lambda$-eigenform for this graph. For simplicity, we denote the vertex corresponding to $\cO(n)\oplus\cO$ with level structure $\phi\in\PP^1$ as $c_{n,\phi}$, and $c_0=\cO\oplus\cO$. Using the structure of the graph and checking the $\lambda$-eigenform condition at all vertices, we get
    $$
    f(c_{n,\phi})=
    \begin{cases}
        a\left(\frac{q}{\lambda}\right)^n,&\phi\ne\infty,\\
        a\lambda^n-a(q-1)\sum_{i=0}^{n-1}\lambda^i\left(\frac{q}{\lambda}\right)^{n-i},&\phi=\infty.
    \end{cases}
    $$
    Therefore, $\dim_\lambda\Gamma=1$ for any $\lambda$.
}
\end{example}

\begin{example}\label{ex:nontrivial_monodromy_when_ramified_at_x}
{\rm
    We now describe a graph ramified both at $x$ and another point $y$, showing a non-trivial torus monodromy over $x$-ramification. Let $G=\PGL_2$, $X=\PP^1$, $x,y\in \PP^1(k)$. For simplicity, assume that $x=0$ and $y\ne\infty$. We will only describe the cusp.

Denote by $\cE_n$ the bundle $\cO(n)\oplus\cO$. Level structures on $\cE_n$ are given by
$$
\PGL_2(k)\times \PGL_2(k)/(T(k)\ltimes U(k)^2)\simeq \PGL_2(k)/B(k)\times \PGL_2(k)/U(k).
$$
Using this, we will represent level structures on $\cE_n$ as pairs $(p,a)$, where $p\in\PP^1(k)$ is a level structure at $x$ and $a\in \PGL_2(k)$ represents the $y$-level structure $aU(k)$. For the level structure at $x$, we represent a point $p\in k$ by $\smat 10p1$ and $\infty$ by $\smat 0110$.

    Recall that Hecke modifications are represented by inclusion matrices
$$
\sigma=\mat{\pi_x}{c}01\text{ or }\mat{1}{0}{0}{\pi_x},\qquad c\in k.
$$
At $y$, these matrices take the form
$$
\sigma=\mat{y}{c}01\text{ or }\mat{1}{0}{0}{y},\qquad c\in k,
$$
 transporting a level structure $a$ at $y$ to $a\sigma$. Transport of the level structure at $x$ depends on the choice of a double coset in $K\Delta K$ and was completely described above. We will look separately how different edges in Figure \ref{fig:PGL2_P1_ramified_example} lift to the $y$-ramified setting. We split them in the following types:
 \begin{itemize}
     \item {\bfseries Type I:} edges from $\infty$ to $\infty$.
     \item {\bfseries Type II:} edges from $\infty$ to $\PP^1\setminus\{0,\infty\}$.
     \item {\bfseries Type III:} edges from $\PP^1\setminus\{0,\infty\}$ to $0$.
     \item {\bfseries Type IV:} edges from $0$ to $0$.
 \end{itemize}
We describe them one-by-one. 

{\bfseries Types I-II:} Recall that there is precisely one inclusion matrix $\sigma$ for all edges from a given bundle with a level structure. In this case, 
$$
\sigma=\mat100{\pi_x}.
$$
Computation (\ref{eq:ramified_connections_of_infinity}) yields connections
\begin{align*}
    \text{Type I}:\,(\infty,a)&\stackrel{\smat1{0}01}\longrightarrow\left(\infty,a\smat100y\right),\\
    \text{Type II}:\,(\infty,a)&\stackrel{\smat1{-\pi c^{-1}}01}\longrightarrow\left(\smat {c^{-1}}110,a\smat100y\right)\stackrel{\smat{c}10{-c^{-1}}}\sim\left(c,a\smat{c}10{-c^{-1}y}\right)\sim \left(c,a\smat{c}00{-c^{-1}y}\right),
\end{align*}
where we indicated which double coset in $K\Delta K$ corresponds to these edges and which matrix in $T(k)\ltimes U(k)$ we are using to bring the level structure at $x$ to the form $\smat10p1$ or $\smat0110$. In the last identification, we used the group $U(k)$ identifying level structures over $y$.

{\bfseries Types III-IV:} Computation (\ref{eq:ramified_connections_away_from_infinity}) yields $\sigma=\smat{\pi_x}001$ and connections
\begin{align*}
    (p,a)&\stackrel{\smat1{\pi c}01}\longrightarrow\left(\smat1{c}01,a\smat y001\right)\stackrel{\smat1{-c}01}\sim\left(0,a\smat{y}{-yc}0{1}\right)\sim \left(0,a\smat{y}{0}0{1}\right).
\end{align*}
Since we are over $\PGL_2$, all the level structures are defined up to rescaling. Setting $y=1$ and setting $t_c:=\smat c001\in T(k)$ for $c\in k^\times$, we can summarize our connections as follows:
\begin{align*}
    (\infty,a)&\longrightarrow (\infty,a),\\
    (\infty,a)&\longrightarrow (c,at_{-c^2}),\\
    (p,a)&\longrightarrow (0,a).
\end{align*}
So, only Type II edges change the representative of the level structure at $y$. We can organize the change in a simple way as follows. Let $g$ be a primitive generator $k^\times$ (recall that $k$ is a finite field, hence $k^\times$ is cyclic). Then Type II edges take the form
$$
(\infty,a)\longrightarrow 
\begin{cases}
    (g^k,at_{g^{2k+\frac{q-1}{2}}}),&2\nmid q,\\
    (g^k,at_{g^{2k}}),&2\mid q.
\end{cases}
$$
See Figure \ref{fig:PGL2_P1_ramified_at_x_and_y} for an example when $k=\FF_4$. One can easily find a (topological) loop on the graph downstairs which lifts to a non-loop in the graph upstairs.
 }
\end{example}

\begin{figure}
\centering
\scalebox{.9}{

\includegraphics{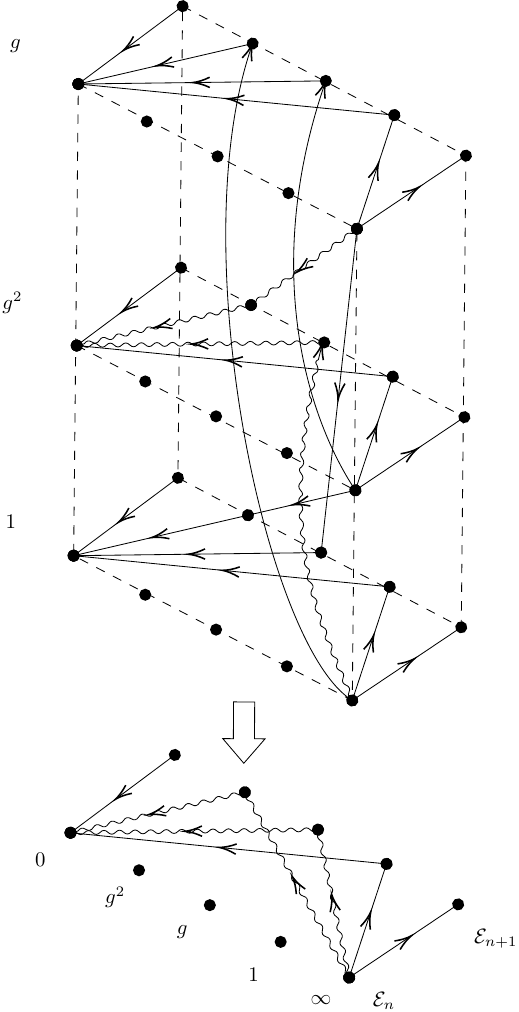}

}
\caption{Connected component of the cusp of the ramified graph for $G=\PGL_2$ and $X=\PP^1$ at $[x]+[y]$, $k=\FF_4$. Here, $g\in k$ satisfies $g^2=g+1$ and the vertical arrow forgets the level structure at $y$. Edge multiplicities are the same as on Figure \ref{fig:PGL2_P1_ramified_example}. Waved edges represent an example of a loop in the graph downstairs which lifts to a non-loop in the graph upstairs.}
\label{fig:PGL2_P1_ramified_at_x_and_y}
\end{figure}

We finish this subsection with generalizing the above results to the equation of the form 
$$
(\Phi-\lambda)f=g.
$$
\begin{theorem}\label{t:Phi-lambda_is_surjective}
    In all the above examples, assume that $\lambda\in\CC$ is not an eigenvalue for the adjacency matrix of $\Gamma'$ (see Remark \ref{r:finite_set_of_lambda}). Then the operator 
    $$
    \Phi-\lambda\cdot\mathrm{id}:\CC\Gamma\to \CC\Gamma
    $$
    is surjective. In particular, for any fixed $g\in\operatorname{Fun}(\Bun_G(k))$, the equation
    \begin{equation}\label{eq:Phi-lambda=g}
    (\Phi-\lambda)(f)=g
    \end{equation}
    has a finite-dimensional affine space of solutions whose dimension is given by the lower bound in Theorems \ref{t:PGL2_ramified_not_at_x}, \ref{t:PGL2_ramified_at_x}, and \ref{t:PGL2_ramified_at_x_and_D}.
\end{theorem}
\begin{proof}
    The proof repeats the above arguments with $g$ added in all the equations. To illustrate it, let us show how the proof of Theorem \ref{t:eigenforms_for_graphs_with_decomposition} modifies (which is sufficient because of Remark \ref{r:ramified_case_is_stictly_propagative}). Using the notation from that proof and thinking of $g$ as a vector $(g',g_1,g_2,\ldots)$, we translate equation (\ref{eq:Phi-lambda=g}) into
    \begin{itemize}
        \item $(\lambda-M)v'+g'=Av_1$,
        \item $\lambda v_1+g_1=M_1v_1+Bv'+A_2 v_{2}$,
        \item $\lambda v_i+g_i=M_iv_i+B_iv_{i-1}+A_{i+1} v_{i+1}$.
    \end{itemize}
    Since we assumed that $\lambda$ is not an eigenvalue of $M$, the first equation has a unique solution in $v'$ for fixed $v_1$. The second equation has a unique solution in $v_2$, and the third equation has a unique solution in $v_{i+1}$. So, we are done.
\end{proof}

From this, we can immediately describe the generalized eigenspaces of $\Phi$. For $k\ge 1$, denote
$$
\dim_{\lambda,k}\Gamma:=\dim\frac{\ker(\Phi-\lambda)^k}{\ker(\Phi-\lambda)^{k-1}}.
$$
\begin{corollary}\label{cor:generalized_eigenspace_dimension}
    Keep the assumptions of Theorem \ref{t:Phi-lambda_is_surjective}. Then
    $$
    \dim_{\lambda,k}\Gamma=\dim_\lambda\Gamma.
    $$
    In particular,
    $$
    \dim\ker(\Phi-\lambda)^k=k\cdot \dim_{\lambda}\Gamma.
    $$
\end{corollary}
\begin{proof}
    We have an exact sequence
    $$
    0\to \ker(\Phi-\lambda)\to \ker(\Phi-\lambda)^k\xrightarrow{\Phi-\lambda}\ker(\Phi-\lambda)^{k-1}.
    $$
    By Theorem \ref{t:Phi-lambda_is_surjective}, the rightmost arrow is surjective. This proves the claim.
\end{proof}
\begin{remark}
    By the proof of Theorem \ref{t:Phi-lambda_is_surjective}, if $g$ satisfies the moderate growth condition, then all functions in $(\Phi-\lambda)^{-1}(g)$ do. In particular, all elements in $\ker(\Phi-\lambda)^k$ are automorphic forms. Moreover, using the same arguments as in Corollary \ref{cor:vector_bundle_of_eigenforms} and Remark \ref{r:vector_bundle_of_eigenforms_in_ramification}, we can see that such automorphic forms form a trivial vector bundle over a base in $\CC$.
\end{remark}

\subsection{Jets of eigenfamilies and explicit Jordan bases}
\label{ss:jets_of_eigenfamilies}

The preceding corollary determines the dimensions of the generalized
eigenspaces. We now refine it by constructing explicit Jordan bases following the unramified strategy in \cite[Section 11]{Lor12tor}.

Fix one of the $\PGL_2$ Hecke graphs considered above. Let $\Gamma'$ be
the finite subgraph used in the corresponding eigenspace computation,
let $M$ be its adjacency matrix, and set
\[
R:=
\left\{
\lambda\in\CC^\times:
\det(\lambda\operatorname{id}-M)\neq0
\right\}.
\]
By Corollary~\ref{cor:vector_bundle_of_eigenforms} in the unramified case
and Remark~\ref{r:vector_bundle_of_eigenforms_in_ramification} in the
ramified cases, the spaces
\[
\ker(\Phi-\lambda),
\; \lambda\in R,
\]
are the fibers of a trivial algebraic vector bundle of some rank $N$.
Choose an algebraic frame
$
F_1(\lambda),\ldots,F_N(\lambda)\in\CC\Gamma,
\; \lambda\in R.
$
Thus
$
\Phi F_j(\lambda)=\lambda F_j(\lambda),
$
and
$
F_1(\lambda),\ldots,F_N(\lambda)
$
form a basis of $\ker(\Phi-\lambda)$ for every $\lambda\in R$.
Here algebraicity means that, for every vertex $v\in\Gamma$, the function
$
\lambda\longmapsto F_j(\lambda)(v)
$
is regular on $R$.

\begin{theorem}[Explicit Jordan bases from jets]
\label{t:explicit_Jordan_bases_from_jets}
Fix $\lambda_0\in R$. For $i\geq0$, define
\[
F_j^{[i]}(\lambda_0)
:=
\left.
\frac{1}{i!}
\frac{d^i}{d\lambda^i}F_j(\lambda)
\right|_{\lambda=\lambda_0}.
\]
Then
$
(\Phi-\lambda_0)F_j^{[0]}(\lambda_0)=0
$
and
$
(\Phi-\lambda_0)F_j^{[i]}(\lambda_0)
=
F_j^{[i-1]}(\lambda_0),
\; i\geq1.
$
Consequently, for every $m\geq1$, the vectors
$
\left\{
F_j^{[i]}(\lambda_0):
1\leq j\leq N,\ 0\leq i<m
\right\}
$
form a basis of
$
\ker\bigl((\Phi-\lambda_0)^m\bigr).
$
More precisely,
\[
\ker\bigl((\Phi-\lambda_0)^m\bigr)
=
\bigoplus_{j=1}^N
\operatorname{span}_{\CC}
\left\{
F_j^{[0]}(\lambda_0),\ldots,
F_j^{[m-1]}(\lambda_0)
\right\}.
\]
Thus the restriction of $\Phi-\lambda_0$ to
$\ker\bigl((\Phi-\lambda_0)^m\bigr)$ is the direct sum of $N$ nilpotent
Jordan blocks of size $m$.
\end{theorem}

\begin{proof}
Since $\Gamma$ is locally finite, every coordinate of $\Phi f$ is a
finite linear combination of coordinates of $f$. We may therefore
differentiate the identity
$
\Phi F_j(\lambda)=\lambda F_j(\lambda)
$
coordinatewise. Differentiating $i$ times, evaluating at $\lambda_0$,
and dividing by $i!$ gives
$
(\Phi-\lambda_0)F_j^{[i]}(\lambda_0)
=
F_j^{[i-1]}(\lambda_0),
\; i\geq1.
$
In particular,
$
(\Phi-\lambda_0)^{i+1}F_j^{[i]}(\lambda_0)=0.
$
Hence all the displayed vectors with $0\leq i<m$ belong to
$\ker\bigl((\Phi-\lambda_0)^m\bigr)$.

Put $A=\Phi-\lambda_0$. Suppose that
\[
\sum_{i=0}^{m-1}\sum_{j=1}^N
c_{i,j}F_j^{[i]}(\lambda_0)=0
\]
is a nontrivial linear relation, and let $s$ be maximal such that
$c_{s,j}\neq0$ for some $j$. Applying $A^s$ gives
$
\sum_{j=1}^N c_{s,j}F_j(\lambda_0)=0,
$
because
$
A^sF_j^{[s]}(\lambda_0)=F_j(\lambda_0),
\;
A^sF_j^{[i]}(\lambda_0)=0
\;\text{for }i<s.
$
This contradicts the fact that
$
F_1(\lambda_0),\ldots,F_N(\lambda_0)
$
form a basis of $\ker A$. Hence the displayed $mN$ vectors are linearly
independent.

Since $\lambda_0\in R$,
Corollary~\ref{cor:generalized_eigenspace_dimension} applies and gives
$
\dim\ker(A^m)
=
m\dim\ker A
=
mN.
$
Therefore the $mN$ independent vectors form a basis of $\ker(A^m)$.
\end{proof}

\subsection{Spaces of eigenforms for $\PGL_n$} \label{PGLn case}
Let $X=\PP^1$ and $x$ be a point of $X$ of degree $d$. Any vector bundle over $\PP^1$ splits into a direct sum of line bundles, and we will write such a bundle in the form $\cE=\cL_1\oplus\ldots\oplus\cL_n$, where $\deg\cL_i\ge\deg\cL_{i-1}$. Note that bundles of the same degree $e$ are isomorphic to $\cO(e)$, so such form is unique. We associate to it a sequence of positive integers
$$
l(\cE)=(l_1,\ldots,l_{n-1})=(\deg \cL_1-\deg \cL_2,\deg \cL_2-\deg \cL_3,\ldots,\deg \cL_{n-1}-\deg \cL_n).
$$
A Hecke modification with respect to $w_1$ at $x$ twists one of these line bundles by $-x$. It modifies $l(\cE)$ by adding
\begin{align*}
    (-d,0,\ldots,0),\quad (d,-d,0,\ldots,0),\quad\ldots\quad (0,\ldots,0,d,-d),\quad (0,\ldots,0,d),
\end{align*}
corresponding to twisting $\cL_1,\ldots,\cL_n$, respectively. Note that only te last modification increases $l_{n-1}$. Therefore, it is reasonable to define
$$
\Gamma_i:=\{\cE:di\le l_{n-1}(\cE)<d(i+1)\},\quad i=0,1,\ldots
$$
The decomposition is strictly propagative with $\Gamma'=\varnothing$, so we can apply Theorem \ref{t:eigenforms_for_graphs_with_decomposition}. Therefore, the space of $\lambda$-eigenfunctions on $\Gamma$ is in bijection with functions on $\Gamma_1$, and the space of $\lambda$-eigenforms is in bijection with functions on $\Gamma_1$ of moderate growth. In particular, the space is uncountably-dimensional.

\begin{example}
{\rm Let $\deg x=1$ and $n=3$. In \cite[Example 5.3]{Alv19}, the graph with all edge multiplicities was explicitly computed, and is presented on Figure \ref{fig:PGL3_unramified}.
\begin{figure}
\centering
\scalebox{.9}{
\includegraphics{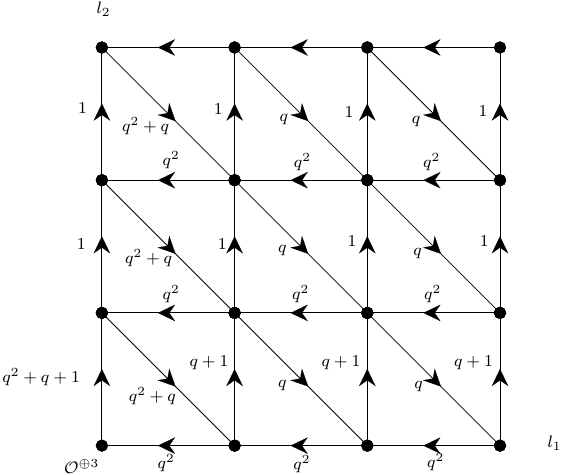}
}
\caption{Unramified graph for $\PGL_3$.}
\label{fig:PGL3_unramified}
\end{figure}
Here, $\Gamma_i$ is the horizontal line with the $l_2$-coordinate $i$.
}
\end{example}

\subsection{Connections to the theory of Eisenstein series}
We connect our findings to the known theory of Eisenstein series. For more detailed explanations, we refer to \cite{KZ2}. Let $K:=G(\cO)$ and $K':=K(D)$. In \cite{Li79}, where the dimensions of spaces of Eisenstein series were computed, the following vector space is defined:
$$
V^U_0:=\{v\in\CC[K/K']:\forall_{b\in T(\FF_q)U(\cO)}\, bv=v\}.
$$
Then we have
\begin{theorem}[\cite{Li79}, Theorem 7.1]\label{t:Li_dimension_formula}
    The space of all $\lambda$-eigenforms decomposes as the direct sum of the spaces of Eisenstein series and cusp forms. Moreover, the subspace of Eisenstein series has dimension
    $$
    \deg x\cdot |\Pic(X)(\FF_q)|\cdot\dim V^U_0.
    $$
\end{theorem}

We compare this formula to the one found in Theorem \ref{t:PGL2_ramified_not_at_x}. We have
\begin{align*}
    V_0^U&=\{v\in V^U:\forall_{a,d\in\FF_q^\times}\smat a00d v=v\}\\
    &=\{v\in\CC[K/K']:\forall_{b\in T(\FF_q)U(\cO)}\, uv=v\}\\
    &=\mathrm{Fun}(T(\FF_q)U(\cO_x)\backslash G(\cO_x)/K'_x)\\
    &=\mathrm{Fun}(T(\FF_q)U(\cO_D)\backslash G(\cO_D)).
\end{align*}

Note that we get the same formula as in Theorems \ref{t:PGL2_ramified_not_at_x} and \ref{t:Li_dimension_formula}. However, the approach in \cite{Li79} does not tackle the case when the Hecke operator is taken in a ramified point. Theorems \ref{t:PGL2_ramified_at_x} and \ref{t:PGL2_ramified_at_x_and_D} provide dimension formulas in this case, which are clearly different from the unramified dimensions. Therefore, the Eisenstein series construction does not immediately give eigenforms for $\Phi_x$ when $x$ is ramified. In future work, we will define Eisenstein series in the ramified case which will give bases to spaces of eigenforms for generic eigenvalues.

\newpage

\section{General Reductive Case} \label{s:general_G}
\subsection{Notations and general assumptions}
As usual, we use notation of Section \ref{s:preliminaries_on_Hecke_graphs}. Additionally, we introduce the following notation,
compatible with \cite{Schieder2015}:
\begin{itemize}
    \item $T$ is a fixed split maximal torus of $G$, and $B$ is a fixed Borel subgroup of $G$
    containing $T$.
    \item $U$ is the unipotent radical of $B$.
    \item $\Phi=\Phi(G,T)$ is the root system of $G$ with respect to $T$, $W$ is the Weyl group of $G$.
    \item $\Phi^+=\Phi^+(G,T,B)$ and $\Phi^-=-\Phi^+$ are the sets of positive and negative roots.
    \item $\alpha_1,\ldots,\alpha_n$ and $\check\alpha_1,\ldots,\check\alpha_n$ are the simple roots
    and simple coroots, respectively.
    \item For a root $\alpha\in \Phi$, $U_\alpha$ is the corresponding root subgroup of $G$.
    \item $\check\Lambda_G:=X_*(T)$ and $\Lambda_G:=X^*(T)$ are the coweight and weight lattices,
    respectively.
    \item For $\chi\in X^*(T)$, we define $k_{\chi}$ to be the corresponding $1$-dimensional representation of $T$.
\end{itemize}

For a root $\alpha\in \Phi$ and a coweight $\nu\in X_*(T)$, we write
\[
\langle \alpha,\nu\rangle\in \mathbb Z
\]
for the natural pairing.

Let $P$ be a parabolic subgroup of $G$ containing $B$. We define the following notions:
\begin{itemize}
    \item $I_P\subset \{1,\ldots,n\}$
    is the set of $i$ such that
    $
    \fg_{-\alpha_i}\subset \operatorname{Lie}(P).
    $
    This set characterizes $P$ uniquely.
    \item $U(P)$ is the unipotent radical of $P$, and
    $
    M:=P/U(P)
    $
    is the Levi quotient.
    \item We write $\check\Lambda_M$ and $\Lambda_M$ for the coweight and weight lattices of $M$.
     \item For a character $\chi:P\to \mathbb G_m$, we also write $k_\chi$ for the corresponding
    $1$-dimensional representation of $P$.
    \item A character $\chi:P\to \mathbb G_m$ is called {\bfseries dominant} if its restriction to $T$
    is a nonnegative integral combination of the simple roots $\alpha_i$ with $i\notin I_P$.
    \item We write
    \[
    \Lambda_{G,P}:=\{\chi:P\to \mathbb G_m \mid \chi|_{Z(G)^\circ}=1\},
    \qquad
    \check\Lambda_{G,P}:=\operatorname{Hom}_{\mathbb Z}(\Lambda_{G,P},\mathbb Z).
    \]
    \item If $\mathcal P_P$ is a principal $P$-bundle on $X$, its degree is the element
    \[
    \deg_P(\mathcal P_P)\in \check\Lambda_{G,P}
    \]
    characterized by
    \[
    \langle \chi,\deg_P(\mathcal P_P)\rangle
    :=
    \deg\bigl(\mathcal P_P\times^P k_\chi\bigr)
    \qquad\text{for all }\chi\in \Lambda_{G,P}.
    \]
    \item An element $\check\lambda_P\in \check\Lambda_{G,P}$ is called {\bfseries dominant} if
    \[
    \langle \chi,\check\lambda_P\rangle\ge 0
    \]
    for every dominant character $\chi\in \Lambda_{G,P}$, and it is called
    {\bfseries dominant $P$-regular} if
    \[
    \langle \chi,\check\lambda_P\rangle>0
    \]
    for every nontrivial dominant character $\chi\in \Lambda_{G,P}$.
\end{itemize}

\begin{definition}
A principal $G$-bundle $\cP_G$ on $X$ is called {\bfseries semi-stable} if for every parabolic
subgroup $P\subset G$, every reduction $\cP_P$ of $\cP_G$ to $P$, and every dominant character
$\chi:P\to \mathbb G_m$ whose restriction to $Z(G)^\circ$ is trivial, one has
\[
\deg\bigl(\cP_P\times^P k_\chi\bigr)\le 0.
\]
\end{definition}

\begin{theorem}[Generalizes \ref{t:HN_filt}] \label{HN reduction}
Let $\mathcal P_G$ be a principal $G$-bundle on $X$. Then there exists a unique parabolic
subgroup $P\subset G$, a unique dominant $P$-regular element $\check\lambda_P\in \check\Lambda_{G,P}$, and a unique reduction $\mathcal P_P$ of $\mathcal P_G$ to $P$ such that 
\[
\mathcal P_P\in\Bun^{ss}_{P,\check\lambda}
:=
\left\{
\mathcal P_P \in \Bun_P
\;\middle|\;
\deg(\mathcal P_P)=\check\lambda
\text{ and }
\mathcal P_P/U(P)\text{ is semi-stable}
\right\}.
\]

Equivalently, if $M:=P/U(P)$ is the Levi quotient, then the induced $M$-bundle
$\mathcal P_M:=\mathcal P_P/U(P)$
is semi-stable, and the degree of the reduction $\mathcal P_P$ is dominant $P$-regular.
\end{theorem}

\begin{proof}
Follows from \cite[Theorem 2.3.3(a),(c) and Remark 2.4.1]{Schieder2015}.

\end{proof}

In the situation of the above theorem, we call $\cP_P$ the {\bfseries Harder-Narasimhan (HN) reduction of $\cP_G$}.
\begin{corollary}\label{cor:Bcase-strictly-dominant} 
Assume the standing notation. Let $\mathcal P_B$ be a $B$-reduction of a $G$-bundle $\mathcal P_G$, and write
$\mathcal P_T:=\mathcal P_B/U$. For $\alpha\in \Phi^+$ set
$\mathcal L_\alpha:=\mathcal P_T\times^T k_\alpha$.
If $\deg(\mathcal L_\alpha)>0$ for all $\alpha\in \Phi^+$, then the Harder--Narasimhan parabolic of $\mathcal P_G$
is $B$ and $\mathcal P_B$ is the canonical Harder--Narasimhan reduction.
\end{corollary} 

\begin{proof}
Since the Levi quotient of $B$ is $T$, the induced Levi bundle $\mathcal P_T:=\mathcal P_B/U$ is a $T$-bundle, hence is semi-stable. Now, let $\alpha_1,\dots,\alpha_n$ be the simple roots. Since each simple root is a positive root, the hypothesis implies
\[\deg(\mathcal L_{\alpha_i})>0
\qquad\text{for all } i=1,\dots,n.\]
Since dominant $B$-regularity is equivalent to strict positivity on all simple roots,
the above inequalities show that the degree of the $B$-reduction $\mathcal P_B$ is
dominant $B$-regular.
Hence $\mathcal P_B$ satisfies the defining properties of the canonical
Harder--Narasimhan reduction from Theorem~\ref{HN reduction}.

By the uniqueness statement in Theorem~\ref{HN reduction}, it follows that
$\mathcal P_B$ is the canonical Harder--Narasimhan reduction of $\mathcal P_G$.
Therefore the Harder--Narasimhan parabolic of $\mathcal P_G$ is $B$.

\end{proof}

We write $\mathrm{Bun}_G$ for the moduli stack of $G$-bundles on $X$, and
$\mathrm{Bun}_{G,D}$ for the moduli stack of $G$-bundles equipped with a $D$-level structure.
We will study the forgetful morphism
\[
p_{D,D_2}:\mathrm{Bun}_{G,D}\longrightarrow \mathrm{Bun}_{G,D_2},
\]
where $D=D_1+D_2$ satisfying (\ref{eq:d1_d2_condition}).

 Let $\cP_G$ be a $G$-bundle on $X$. When we say that {\bfseries $\cP_G$ has HN reduction to $B$},
we mean that the canonical Harder--Narasimhan reduction of $\cP_G$ is a reduction $\cP_B$ to the Borel subgroup $B$. In that case, we write
\[
\cP_T:=\cP_B/U
\]
for the induced $T$-bundle, and for each $\alpha\in \Phi^+$ we set
\[
\cL_\alpha:=\cP_T\times^T k_\alpha,
\]
where $k_\alpha$ is the one-dimensional $T$-representation of weight $\alpha$.

As in the case of $\GL_n$, we will be interested in the cusp locus of $\Bun_G$.

\begin{definition}[The $D$-cusp locus]\label{def:Dcusp}
A $G$-bundle $\cP$ on $X$ is {\bfseries in the $D$-cusp} if it has HN reduction to $B$ and satisfies
\[
\deg(\cL_\alpha)>2g-2+\deg(D)\qquad\text{for all }\alpha\in \Phi^+.
\]
The loci of such bundles (with level structures) in $\mathrm{Bun}_G$ and $\mathrm{Bun}_{G,D}$ are denoted by 
$\mathrm{Bun}_{G}^{D\text{-cusp}}$ and $\mathrm{Bun}_{G,D}^{D\text{-cusp}}$, respectively.
\end{definition}

As in Section \ref{s:preliminaries_on_Hecke_graphs}, fix a dominant coweight $\mu\in X_*(T)$.

\begin{definition}[Cusp locus]\label{def:deepcusp}
A $G$-bundle $\cP\in \mathrm{Bun}_{G}(k)$ is {\bfseries in the $(D,x,\mu)$-cusp}
if $\cP$ has HN reduction to $B$ and
\[
\deg(\cL_\alpha)>2g-2+\deg(D)+\langle \alpha,w(\mu)\rangle\,\deg(x)
\qquad\text{for all }\alpha\in \Phi^+, w\in W.
\]
The loci of such bundles (with level structures) in $\mathrm{Bun}_G$ and $\mathrm{Bun}_{G,D}$ are denoted respectively by 
$\mathrm{Bun}_{G}^{D\text{-cusp}}(D,x,\mu)$ and $\mathrm{Bun}_{G,D}^{D\text{-cusp}}(D,x,\mu)$.
\end{definition}

\begin{definition}[$P$-compatible trivialization along $D$]\label{def:Bcompatible}
Let $\cP_P$ be the HN reduction of $\cP_G$. A $D$-level structure
\[
\psi:\cP_G|_D\xrightarrow{\sim} G\times \operatorname{Spec}(\mathcal O_D)
\]
is called {\bfseries $P$-compatible} if it carries the induced $P$-reduction $\cP_P|_D$
to
\[
P\times \operatorname{Spec}(\mathcal O_D)\subset G\times \operatorname{Spec}(\mathcal O_D).
\]
\end{definition}

We use the shorthand
\[
G(\mathcal O_D),\qquad B(\mathcal O_D),\qquad U(\mathcal O_D),\qquad T(\mathcal O_D)
\]
for the $\mathcal O_D$-points. We regard $T(k)\subset T(\mathcal O_D)$ via constant sections.

\subsection{Generalizations of the $\GL_n$ results}

\begin{lemma}[A height filtration on $U$]\label{lem:height_filtration_U}
Assume the choices of $(G,B,T,U,\Phi^+,\Delta)$ and the standing notation  from the
previous section, and let $\mathrm{ht}(\alpha)$ denote
the height of $\alpha$ with respect to $\Delta$.

For each integer $m\ge 1$, let $U_{\ge m}\subset U$ be the closed subgroup scheme
generated by the root subgroups $U_\alpha$ for $\alpha\in\Phi^+$ with
$\mathrm{ht}(\alpha)\ge m$. Then

\begin{enumerate}
\item[(a)] $U_{\ge m}$ is a closed, normal, $T$-stable subgroup scheme of $U$; moreover
$U_{\ge 1}=U$, and $U_{\ge m}=1$ for $m$ larger than the maximal height among $\Phi^+$.
\item[(b)] For each $m$, the quotient $U_{\ge m}/U_{\ge m+1}$ is a commutative (vector)
group scheme, and there is a canonical $T$-equivariant isomorphism of group schemes
\[
U_{\ge m}/U_{\ge m+1}\ \cong\ \prod_{\substack{\alpha\in\Phi^+\\ \mathrm{ht}(\alpha)=m}} U_\alpha,
\]
where each $U_\alpha\simeq \mathbb G_a$ and $T$ acts on $U_\alpha$ via the character $\alpha$.
\end{enumerate}
\end{lemma}

\begin{proof}
 This is the standard height filtration on the unipotent radical.
By the Chevalley commutator/conjugation formula \cite[\S Exp.~XXII, Cor.~5.5.4]{SGA3III},
we have $[U_\alpha,U_\beta]\subset\langle U_{i\alpha+j\beta}\;:\; i,j\ge 1\rangle$,
hence commutators strictly raise height. The conclusions follow.

\end{proof}

Recall that $D=D_1+D_2$
is a decomposition into effective divisors satisfying (\ref{eq:d1_d2_condition}).
\begin{theorem}[Generalization of Theorem 4.9]\label{thm:deep_cusp_reductive}
Let \(P_G\) be a principal \(G\)-bundle on \(X\) which is \(D\)-cusp, and let
\[
\varphi_{D_2}:P_G|_{D_2}\xrightarrow{\sim}G\times D_2
\]
be a \(D_2\)-level structure. If \(D_2=0\), this means that no retained level structure is fixed. Then:

\begin{enumerate}
\item The canonical Harder--Narasimhan \(B\)-reduction \(P_B\) is induced from
\[
P_T:=P_B/U,
\]
i.e.
\[
P_B\simeq P_T\times^T B.
\]

\item Fix a \(B\)-compatible trivialization
\[
\psi:P_G|_{D_1}\xrightarrow{\sim}G\times D_1.
\]
Under the induced identification
\[
\Aut(P_G|_{D_1})\simeq G(O_{D_1}),
\]
the image of the restriction homomorphism
\[
\operatorname{res}_{D_1}:
\Aut(P_G,\varphi_{D_2})
\longrightarrow
\Aut(P_G|_{D_1})\simeq G(O_{D_1})
\]
is exactly
\[
\begin{cases}
T(k)\ltimes U(O_{D_1}), & D_2=0,\\
U(O_{D_1}), & D_2\neq 0.
\end{cases}
\]
\end{enumerate}
\end{theorem}

\begin{proof}
Because the Harder--Narasimhan reduction $\mathcal P_B$ is canonical and unique, every automorphism of
$\mathcal P$ preserves it. Hence restriction induces an identification
\[
\mathrm{Aut}(\mathcal P)\ =\ \mathrm{Aut}(\mathcal P_B).
\]

\smallskip

Proof of (i). Consider the exact sequence of group schemes
\[
1 \to U \to B \to T \to 1
\]
and let
\[
P_T := P_B/U.
\]
After fixing \(P_T\), isomorphism classes of \(B\)-bundles inducing \(P_T\) are classified by
the pointed set
\[
H^1_{\mathrm{fppf}}(X,U_{P_T}),
\qquad
U_{P_T}:=P_T\times^T U;
\]
see \cite[III.2.4.2]{GiraudCNA}. The given reduction \(P_B\) therefore determines a class in
\(H^1_{\mathrm{fppf}}(X,U_{P_T})\), so it suffices to prove
\[
H^1_{\mathrm{fppf}}(X,U_{P_T})=\{\ast\}.
\]

Let $U_{\ge m}$ be the height filtration from Lemma \ref{lem:height_filtration_U}, and put
$\cU_m := (U_{\ge m})_{P_T} = P_T \times^T U_{\ge m}$ and $\cQ_m := U_m/U_{m+1}$. Then Lemma~\ref{lem:height_filtration_U}(b) gives
\[
\cQ_m \cong \prod_{\substack{\alpha\in \Phi^+\\ \mathrm{ht}(\alpha)=m}} \mathbb G_a(\cL_\alpha),
\qquad
\cL_\alpha:=P_T\times^T k_\alpha.
\]
Since \(P\) is \(D\)-cusp, we have \(\deg(\cL_\alpha)>2g-2\) for every \(\alpha\in\Phi^+\). Hence
\[
H^1(X,\cL_\alpha)=0
\]
by Serre duality, and therefore
\[
H^1_{\mathrm{fppf}}(X,\cQ_m)=0
\]
for all \(m\).

Now apply the exact sequence of pointed sets associated to
\[
1\to \cU_{m+1}\to \cU_m\to \cQ_m\to 1,
\]
see \cite[III.3.3.1]{GiraudCNA}:
\[
H^1_{\mathrm{fppf}}(X,\cU_{m+1}) \longrightarrow
H^1_{\mathrm{fppf}}(X,\cU_m) \longrightarrow
H^1_{\mathrm{fppf}}(X,\cQ_m).
\]

Since the last pointed set is trivial, every class in
\(H^1_{\mathrm{fppf}}(X,\cU_m)\) comes from a class in
\(H^1_{\mathrm{fppf}}(X,\cU_{m+1})\). For \(m\gg 0\), we have \(\cU_m=1\), hence
\[
H^1_{\mathrm{fppf}}(X,\cU_m)=\{\ast\}.
\]
Descending induction on \(m\) yields
\[
H^1_{\mathrm{fppf}}(X,\cU_1)=H^1_{\mathrm{fppf}}(X,\cU)=\{\ast\}.
\]
Therefore the class of \(\cP_B\) is trivial, and
\[
\cP_B \cong \cP_T\times^T B.
\]

\smallskip
\noindent\emph{Proof of (ii).}
Extend $\psi$ to a $B$-compatible trivialization $\tilde\psi$ over $D$. We will first prove that the image of $\Aut(\cP_G)$ in $\Aut(\cP_G|_D)$ is precisely $T(k)\ltimes U(\cO_D)$.

From $(i)$, we identify $\mathcal P_B\simeq \mathcal P_T\times^T B$. Since the HN reduction is unique, any automorphism of $\cP_G$ fixes $\cP_B$, which implies that $\Aut(\cP_G)\simeq \Aut(\cP_B)$.
The group $\Aut(\cP_B)$ is the group of global sections of the adjoint group scheme
\[
\mathrm{Aut}(\mathcal P_B)=H^0\bigl(X,\mathcal P_B\times^{B,\text{ad}} B\bigr)=H^0\bigl(X,\mathcal P_T\times^{T,\text{ad}} B\bigr).
\]
Since $B=T\ltimes U$ and $T$ is abelian, conjugation of $T$ on itself is trivial, hence
$\cP_T\times^{T,\text{ad}} T \cong T\times X$.
Since $X$ is a smooth projective curve and $T$ is affine,
we get $\Hom(X,T)=T(H^0(X,\mathcal O_X))=T(k)$.
Consequently,
\[
\mathrm{Aut}(\cP_B)\cong T(k)\ltimes H^0(X,U_{\cP_T}).
\]

Thus it remains to compute the image of the restriction map
\[
H^0(X,U_{\mathcal P_T})\ \longrightarrow\ H^0(D,(U_{\mathcal P_T})|_D)\ \cong\ U(\mathcal O_D),
\]
where the last identification uses the chosen $B$-compatible trivialization on $D$.

We prove that this restriction map is surjective using the height filtration.
With $\mathcal U_m$ and $\mathcal Q_m$ as above, for each $m$ we have an exact sequence of sheaves of groups
\[
1\to \mathcal U_{m+1}\to \mathcal U_m\to \mathcal Q_m\to 1.
\tag{$\dagger_m$}
\]

 We claim, by descending induction on \(m\), that the restriction map
\[
H^0(X,\cU_m)\longrightarrow H^0(D,\cU_m|_D)
\]
is surjective. For \(m\gg 0\), we have \(\cU_m=1\), so there is nothing to prove.

Assume surjectivity holds for \(m+1\). For each root \(\alpha\) with \(\operatorname{ht}(\alpha)=m\), the \(D\)-cusp condition gives
\[
\deg(\cL_\alpha(-D))>2g-2,
\]
hence
\[
H^1(X,\cL_\alpha(-D))=0
\]
by Serre duality. Therefore the restriction map
\[
H^0(X,\cQ_m)\longrightarrow H^0(D,\cQ_m|_D)
\]
is surjective.

Now consider the commutative diagram
\[
\begin{tikzcd}
1 \arrow[r] & H^0(X,\cU_{m+1}) \arrow[r] \arrow[d] & H^0(X,\cU_m) \arrow[r] \arrow[d] & H^0(X,\cQ_m) \arrow[d] \\
1 \arrow[r] & H^0(D,\cU_{m+1}|_D) \arrow[r] & H^0(D,\cU_m|_D) \arrow[r] & H^0(D,\cQ_m|_D).
\end{tikzcd}
\]

The left vertical map is surjective by the induction hypothesis, and the right
vertical map is surjective by the previous paragraph. We now use exactness of
the associated sequence of pointed sets. Let \(s_D\in H^0(D,U_m|_D)\), and let
\(\bar s_D\) be its image in \(H^0(D,Q_m|_D)\). Choose a lift
\(\bar s\in H^0(X,Q_m)\) of \(\bar s_D\). Since
\[
H^1_{\mathrm{fppf}}(X,U_{m+1})=\{*\},
\]
the section \(\bar s\) lifts to some \(s\in H^0(X,U_m)\). Then \(s|_D\) and
\(s_D\) have the same image in \(H^0(D,Q_m|_D)\), so
\[
s|_D^{-1}s_D\in H^0(D,U_{m+1}|_D).
\]
By the induction hypothesis, \(s|_D^{-1}s_D\) lifts to an element
\(u\in H^0(X,U_{m+1})\). Therefore \(su\in H^0(X,U_m)\) restricts to \(s_D\).
This proves the desired surjectivity.\\
 In particular, for \(m=1\) we obtain
\[
H^0(X,\cU) = H^0(X,\cU_1)\twoheadrightarrow H^0(D,\cU_1|_D)=U(\cO_D).
\]On the torus part, the restriction map identifies the image of
\[
\operatorname{Aut}(P_T)=T(k)
\]
inside \(T(\cO_D)\) with \(T(k)\). Hence
\[
\operatorname{Im}\bigl(\operatorname{\Aut}(P_G)\to G(\cO_D)\bigr)
= T(k)\ltimes U(\cO_D).
\]
Now \(\Aut(P_G,\varphi_{D_2})\) is the preimage in \(\Aut(P_G)\) of
\[
G(O_{D_1})\times\{\operatorname{id}\}
\subset
G(O_{D_1})\times G(O_{D_2})=G(O_D).
\]
Hence its image in \(G(O_{D_1})\) is the projection to \(G(O_{D_1})\) of
\[
\bigl(T(k)\ltimes U(O_D)\bigr)
\cap
\bigl(G(O_{D_1})\times\{\operatorname{id}\}\bigr).
\]
If \(D_2=0\), this image is \(T(k)\ltimes U(O_{D_1})\). If \(D_2\neq 0\), then a constant torus element
\(t\in T(k)\) whose restriction to \(D_2\) is the identity must be \(t=1\). Therefore the image is
\(U(O_{D_1})\). This proves (ii).
\end{proof}

\begin{corollary}[Generalization of Corollary \ref{cor:gln_preimage_size}]\label{cor:deep_cusp_fiber_size_reductive}
Let \(k\) be a finite field. Let \(P_G\) be a \(D\)-cusp \(G\)-bundle, and let
\[
p_{D,D_2}:Bun_{G,D}(k)\longrightarrow Bun_{G,D_2}(k)
\]
be the forgetful map. Then \(p_{D,D_2}^{-1}(P_G,\varphi_{D_2})\) is finite and
\[
\#\,p_{D,D_2}^{-1}(P_G,\varphi_{D_2})
=
|T(O_{D_1})|/|T(k)|^{\delta_{D_2,0}}
\cdot
\bigl|(G/B)(O_{D_1})\bigr|.
\]
\end{corollary}

\begin{proof}
Fix one trivialization \(\tau\) of \(\mathcal P_G|_{D_1}\). Then the set of all trivializations of
\(\mathcal P_G|_{D_1}\) is a right torsor under \(G(\mathcal O_{D_1})\), and two trivializations
define isomorphic objects in \(\mathrm{Bun}_{G,D}(k)\) over the same
\((\mathcal P_G,\phi_{D_2})\) if and only if they differ by the restriction of an automorphism of
\((\mathcal P_G,\phi_{D_2})\). Hence, after choosing \(\tau\), we have an identification of sets
\[
p_{D,D_2}^{-1}(\mathcal P_G,\phi_{D_2})
\cong
G(\mathcal O_{D_1})
/
\mathrm{Im}\Bigl(
\mathrm{Aut}(\mathcal P_G,\phi_{D_2})
\to
\mathrm{Aut}(\mathcal P_G|_{D_1})
\Bigr).
\]
By Theorem~\ref{thm:deep_cusp_reductive}(ii), after replacing \(\tau\) by a
\(G(\mathcal O_{D_1})\)-translate, which does not change the cardinality of the orbit set, the image
subgroup is
\[
T(k)\ltimes U(\mathcal O_{D_1})
\quad\text{if }D_2=0,
\]
and
\[
U(\mathcal O_{D_1})
\quad\text{if }D_2\neq 0.
\]
Therefore
\[
\#\,p_{D,D_2}^{-1}(\mathcal P_G,\phi_{D_2})
=
\begin{cases}
\bigl|G(\mathcal O_{D_1})/(T(k)\ltimes U(\mathcal O_{D_1}))\bigr|, & D_2=0,\\[4pt]
\bigl|G(\mathcal O_{D_1})/U(\mathcal O_{D_1})\bigr|, & D_2\neq 0.
\end{cases}
\]

We now compute these two cardinalities. Since
\[
B(\mathcal O_{D_1})
=
T(\mathcal O_{D_1})\ltimes U(\mathcal O_{D_1}),
\]
and all groups involved are finite, we have
\[
\bigl|G(\mathcal O_{D_1})/(T(k)\ltimes U(\mathcal O_{D_1}))\bigr|
=
\bigl|G(\mathcal O_{D_1})/B(\mathcal O_{D_1})\bigr|
\cdot
\bigl|B(\mathcal O_{D_1})/(T(k)\ltimes U(\mathcal O_{D_1}))\bigr|.
\]
The second factor is
\[
\bigl|T(\mathcal O_{D_1})/T(k)\bigr|.
\]
Similarly,
\[
\bigl|G(\mathcal O_{D_1})/U(\mathcal O_{D_1})\bigr|
=
\bigl|G(\mathcal O_{D_1})/B(\mathcal O_{D_1})\bigr|
\cdot
\bigl|B(\mathcal O_{D_1})/U(\mathcal O_{D_1})\bigr|,
\]
and the second factor is
\[
\bigl|T(\mathcal O_{D_1})\bigr|.
\]

It remains to identify the first factor. Since \(\mathcal O_{D_1}\) is a finite product of Artin
local rings, every line bundle on \(\operatorname{Spec}(\mathcal O_{D_1})\) is trivial. Since \(T\) is
split, this gives
\[
H^1_{\mathrm{fppf}}(\operatorname{Spec}(\mathcal O_{D_1}),T)=1.
\]
Moreover, \(\operatorname{Spec}(\mathcal O_{D_1})\) is affine, so
\[
H^1(\operatorname{Spec}(\mathcal O_{D_1}),\mathbb G_a)
=
H^1(\operatorname{Spec}(\mathcal O_{D_1}),\mathcal O_{\operatorname{Spec}(\mathcal O_{D_1})})
=
0.
\]
Since the split unipotent group \(U\) admits a filtration by normal subgroup schemes whose
successive quotients are isomorphic to \(\mathbb G_a\), descending induction on this filtration gives
\[
H^1_{\mathrm{fppf}}(\operatorname{Spec}(\mathcal O_{D_1}),U)=1.
\]
From the exact sequence
\[
1\to U\to B\to T\to 1
\]
it follows that
\[
H^1_{\mathrm{fppf}}(\operatorname{Spec}(\mathcal O_{D_1}),B)=1.
\]
Therefore every \(B\)-torsor on \(\operatorname{Spec}(\mathcal O_{D_1})\) is trivial. Applying this to
the pullback of the \(B\)-torsor \(G\to G/B\) along an \(\mathcal O_{D_1}\)-point of \(G/B\), we obtain
that the natural map
\[
G(\mathcal O_{D_1})/B(\mathcal O_{D_1})
\longrightarrow
(G/B)(\mathcal O_{D_1})
\]
is bijective. Hence
\[
\bigl|G(\mathcal O_{D_1})/B(\mathcal O_{D_1})\bigr|
=
\bigl|(G/B)(\mathcal O_{D_1})\bigr|.
\]

Combining the above equalities gives
\[
\#\,p_{D,D_2}^{-1}(\mathcal P_G,\phi_{D_2})
=
\frac{|T(\mathcal O_{D_1})|}{|T(k)|^{\delta_{D_2,0}}}
\cdot
\bigl|(G/B)(\mathcal O_{D_1})\bigr|.
\]
This proves the corollary.
\end{proof}

\begin{remark}
When $G=\mathrm{GL}_n$, $B$ is the upper triangular Borel, $T$ the diagonal torus, and $\Phi^+$ consists of
$\alpha_{ij}=e_i-e_j$ for $i<j$. If $\mathcal P$ corresponds to a vector bundle
$E\simeq \bigoplus_{i=1}^n \cL_i$ with degree gaps $\deg(\cL_i)-\deg(\cL_{i+1})>2g-2+\deg D$, then
$\cL_{\alpha_{ij}}\simeq \cL_i\otimes \cL_j^{-1}$ and
$\deg(\cL_i)-\deg(\cL_j)>2g-2+\deg D$ for all $i<j$.
In that case Theorem~\ref{thm:deep_cusp_reductive} and Corollary~\ref{cor:deep_cusp_fiber_size_reductive}
specialize to Theorem~\ref{t:gln_image_of_aut_group} and Corollary~\ref{cor:gln_preimage_size} in the $\mathrm{GL}_n$ case.
\end{remark}


\begin{definition}[Relative position at a point]\label{def:relpos}
Let $x\in |X|$ be a closed point. Let $\cP_T$ and $\cP'_T$ be principal $T$-bundles on $X$, and let
\[
\phi:\cP_T|_{X\setminus\{x\}}\xrightarrow{\sim} \cP'_T|_{X\setminus\{x\}}
\]
be an isomorphism. For $\chi\in X^*(T)$, denote by $k_\chi$ the corresponding $1$-dimensional representation of $T$ and set
\[
\cL_\chi:=\cP_T\times^T k_\chi,\qquad \cL'_\chi:=\cP'_T\times^T k_\chi.
\]
We say that $\phi$ has {\bfseries relative position} $\nu\in X_*(T)$ at $x$ if for every $\chi\in X^*(T)$
the induced isomorphism $\cL_\chi|_{X\setminus\{x\}}\simeq \cL'_\chi|_{X\setminus\{x\}}$ extends across $x$
to an isomorphism of line bundles on $X$
\[
\cL'_\chi \ \simeq\ \cL_\chi\bigl(-\langle \chi,\nu\rangle\cdot x\bigr).
\]
Equivalently, after choosing trivializations of $\cP_T$ and $\cP'_T$ over $\Spec(\mathcal O_x)$,
the transition element of $\phi$ lies in the coset
$T(\mathcal O_x)\cdot \pi_x^\nu \cdot T(\mathcal O_x)\subset T(F_x)$.
\end{definition}

\noindent\textit{Standing notation for Proposition~\ref{prop:mu-mod-preserves-BHN}.}
Let $\cP$ be a principal $G$-bundle in the  cusp locus so that its canonical HN $B$-reduction
$\cP_B$ is induced from $\cP_T:=\cP_B/U$ (by Theorem~\ref{thm:deep_cusp_reductive} (i)).
Given $f:\cP|_{X\setminus\{x\}}\xrightarrow{\sim}\cP^1|_{X\setminus\{x\}}$, let $\cP_B^1$ be the
$B$-reduction of $\cP^1$ obtained by transporting the section $X\setminus\{x\}\to \cP/B$
corresponding to $\cP_B$ along $f$ and extending across $x$ (uniquely, since $G/B$ is proper),
and set $\cP_T^1:=\cP_B^1/U$.

\smallskip
\begin{proposition}[Split reductive $G$ version of Proposition~\ref{p:cusp_connected_to_cusp}]\label{prop:mu-mod-preserves-BHN}
Assume the standing notation above. Let $\nu\in X_*(T)$ be the relative position at $x$
of the induced identification $\cP_T|_{X\setminus\{x\}}\simeq \cP_T^1|_{X\setminus\{x\}}$
(in the sense of Definition~\ref{def:relpos}). Assume that for every $\alpha\in\Phi^+$,
the root line bundle $\cL_\alpha:=\cP_T\times^T k_\alpha$ satisfies
\[
\deg(\cL_\alpha)>\langle \alpha,\nu\rangle\deg(x). \tag{$\dagger$}
\]
Then $\cP_B^1$ is the canonical Harder--Narasimhan $B$-reduction of $\cP^1$.
\end{proposition}

\begin{proof}
\emph{Step 1: extension of the transported reduction.}
Let $s$ be the section of $\mathcal P/B$ defining $\mathcal P_B$. Transporting by $f$ gives
a section $s^\circ$ of $\mathcal P^1/B$ over $X\setminus\{x\}$.
Since $G/B$ is projective (hence proper), $s^\circ$ extends uniquely across $x$ by the valuative criterion
for properness, giving a global section $s'$ and hence a $B$-reduction $\widetilde{\mathcal P}^1_B$ of $\mathcal P^1$.

\emph{Step 2: computation of the induced $T$-bundle.}
Let $\widetilde{\mathcal P}^1_T:=\widetilde{\mathcal P}^1_B/U$.
By construction, $\mathcal P_T$ and $\widetilde{\mathcal P}^1_T$ are identified on $X\setminus\{x\}$ via $f$.
Therefore, for each $\chi\in X^*(T)$ the line bundles $\mathcal L_\chi$ and
$\widetilde{\mathcal L}^1_\chi:=\widetilde{\mathcal P}^1_T\times^T k_\chi$ are isomorphic on $X\setminus\{x\}$,
hence differ by a twist supported at $x$:
\[
\widetilde{\mathcal L}^1_\chi \ \simeq\ \mathcal L_\chi(-m_\chi\cdot x)
\qquad\text{for some }m_\chi\in\mathbb Z.
\]
By the definition of $\nu$ as the relative position at $x$ of the induced identification
\[
\cP_T|_{X\setminus\{x\}} \xrightarrow{\sim} \cP_T^1|_{X\setminus\{x\}},
\]
we have $m_\chi=\langle\chi,\nu\rangle$, hence
\[
\widetilde{\mathcal L}^1_\chi \ \simeq\ \mathcal L_\chi\bigl(-\langle\chi,\nu\rangle\cdot x\bigr)
\qquad\text{for all }\chi\in X^*(T).
\]

\emph{Step 3: $\widetilde{\mathcal P}^1_B$ is the HN reduction.}
For each positive root $\alpha$, we obtain
\[
\deg(\widetilde{\mathcal L}^1_\alpha)=\deg(\mathcal L_\alpha)-\langle\alpha,\nu\rangle\deg(x).
\]
By hypothesis \((\dagger)\) this degree is strictly positive for every $\alpha\in\Phi^+$,, so the degree datum of
$\widetilde{\mathcal P}^1_T$ lies in the interior of the dominant chamber corresponding to $B$.
Since the Levi of $B$ is $T$ (hence $\tilde{\cP^1_{B}}$ is semistable), Theorem~\ref{HN reduction}
implies that $\widetilde{\mathcal P}^1_B$ is the canonical Harder--Narasimhan reduction of $\mathcal P^1$.
This shows that $\widetilde{\mathcal P}^1_B$ is the canonical Harder--Narasimhan reduction of $\mathcal P^1$,
hence the Harder--Narasimhan parabolic of $\mathcal P^1$ is $B$.
\end{proof}

Now, we describe a generalization of the structure in Propositions \ref{p:cusp_connected_to_cusp} to \ref{p:cusp_connections_degrees}. Suppose that two $G$-bundles $\cP_G$ and $\cP_G'$ are connected by the Hecke correspondence $\Phi_x^\mu$ via 
$$
\tau:\cP_G|_{X\setminus x}\simeq \cP_G'|_{X\setminus x}.
$$
Choose a reduction $\cP_B$ of $\cP_G$ to $B$. It corresponds to a section $s$ of $(G/B)_{\cP_G}$. Transporting this section across $\tau$ and extending it uniquely to a section $s'$ of $(G/B)_{\cP_G'}$, we obtain a $B$-reduction $\cP_B'$ of $\cP_G'$. Denote the corresponding $T$-bundles by $\cP_T$ and $\cP_T'$.

We are interested in classifying edges $\tau$ in the Hecke graph by relative positions of $\cP_T$ and $\cP_{T'}$. Recall that $\tau$ is represented by a conjugacy class $\sigma\Delta^\mu K$ in $K\Delta^\mu K$ with $\sigma_y=\Delta_y^\mu=1$ for all $y\ne x$, where we denoted $G(\cO_x)$ by $K$. Arithmetically, reductions $\cP_B$ to $\cP_B'$ correspond to choosing elements $(g_y),(g_y')\in B(\AA)$ representing $\cP_G$ and $\cP_G'$ in $G(F)\backslash G(\AA)/K$. Then the compatibility of these reductions over $\tau$ means that $g_y=g_y'$ for all $y\ne x$ and $g_{x}'g_x^{-1}\in \sigma \Delta^\mu K_x\in G(F_x)/K_x$. The relative position is the reduction of this element in $U(F_x)\backslash G(F_x)/K_x$.

Based on this, we reason in the following way. Recall the Iwasawa decomposition
$$
G(F_x)/K_x=B(F_x)K_x/K_x=\bigsqcup_{\nu\in X_*(T)}U(F_x)\Delta_x^\nu K_x/K_x.
$$
Set $S_\nu:=U(F_x)\Delta^\nu K_x/K_x$ and $\mathrm{Gr}^\mu:=K_x\Delta_x^\mu K_x/K_x$. Then
$$
\mathrm{Gr}^\mu=\bigsqcup_{\nu\in X_*(T)}\mathrm{Gr}^\mu\cap S_\nu,
$$
and we just obtained that the bundles $\cP_G$ and $\cP_G'$ are in relative position $\nu$ if and only if $\sigma_x\Delta_x^\mu K_x\in \mathrm{Gr}^\mu\cap S_\nu$. Using \cite[Theorem 3.2, Remark 3.3]{MV07}, we see that this intersection is non-empty if and only if $\nu$ is a weight of the highest weight representation of the Langlands dual group $\check G$ of highest weight $\mu$. Equivalently, $\nu$ lies in the convex hull of the $W$-orbit of $\mu$ and $\mu-\nu$ lies in the coroot lattice of $G$. Therefore, we have the following result:
\begin{proposition}
    Fix $\cP_G$. The set of edges $\cP_G\to\cP_G'$ in $\Gamma^\mu_x$ of relative position $\nu$ is in bijection with $k_x$-points in $\mathrm{Gr}^\mu\cap S_\nu$. This set is non-empty if and only if $\nu$ lies in the convex hull of the orbit $W\mu$ and $\mu-\nu$ lies in the coroot lattice of $G$. Moreover, if $\cP_G$ and $\cP_G'$ are both $D$-cusp, then $\cP_G'$ is unique up to an isomorphism.
\end{proposition}
\begin{proof}
    What remains is to prove the last statement. Since we assume that $\cP_T$ and $\cP_T'$ are in relative position $\nu$, then $\cP_T'$ is uniquely determined. By Theorem \ref{thm:deep_cusp_reductive}, the extension of $\cP_T'$ to $G$ is isomorphic to $\cP_G'$. Therefore, $\cP_G'$ is unique as well.
\end{proof}

Write \(D=D_1+D_2\) for effective divisors satisfying condition (\ref{eq:d1_d2_condition}).
Let
\[
\Gamma_D:=\Gamma_{D,x}^{\mu},
\qquad
\Gamma_{D_2}:=\Gamma_{D_2,x}^{\mu}.
\]
Let \(\Gamma_D^{\mathrm{cusp}}\subset \Gamma_D\) and
\(\Gamma_{D_2}^{\mathrm{cusp}}\subset \Gamma_{D_2}\) be the full subgraphs spanned by the
\((D,x,\mu)\)-cusp vertices. 
\begin{lemma}[Generalization of Lemma \ref{l:covering_map_GL}]\label{l:covering_map_G}
    Assume that $x\notin\supp D_1$. Then the map $$p_{D,D_2}:\Gamma_{D}^{\mathrm{cusp}}\to \Gamma_{D_2}^{\mathrm{cusp}}$$ is a covering map of graphs.
\end{lemma}
\begin{proof}
    The proof goes along the lines of the proof of Lemma \ref{l:covering_map_GL}, with only modifications in the proof of the claim. In that proof, we replace the embedding $f$ by an isomorphism $f:\tilde \cE|_{X\setminus x}\simeq \cE_{X\setminus x}$. Choose $B$-compatible trivializations of $\cE$ and $\tilde \cE$ over $D$. Since both $f$ and $\alpha$ preserve the HN reduction $\cP_B$, we get
    $$
    f|_{D_1}\in B(\cO_{D_1}),\quad \alpha|_{D_1}\in T(k)U(\cO_{D_1}). 
    $$
    Moreover, if $D_2\ne 0$, then $\alpha|_D\in U(\cO_{D_1})$ by Theorem \ref{thm:deep_cusp_reductive}. Then
    $$
    (f^{-1}\alpha f)|_{D_1}\in
    \begin{cases}
        T(k)U(\cO_{D_1}),&D_2=0,\\
        U(\cO_{D_1}),&D_2\ne 0,
    \end{cases}
    $$
    under a $B$-compatible trivialization of $\cP_G|_{D_1}$. This finishes the proof.
\end{proof}

\begin{theorem}[Generalization of Theorem \ref{t:graph_for_ramification_at_cusp}]
\label{thm:deep_cusp_splitting_reductive}
Assume that \(\operatorname{supp}(D_2)\neq\{x\}\). Then the forgetful map
\[
p_{D,D_2}:\Gamma_D^{\mathrm{cusp}}\longrightarrow \Gamma_{D_2}^{\mathrm{cusp}}
\]
is a disjoint covering of constant degree
\[
N=
|T(\mathcal O_{D_1})|/|T(k)|^{\delta_{D_2,0}}
\cdot
\bigl|(G/B)(\mathcal O_{D_1})\bigr|.
\]
More precisely, there is a decomposition
\[
\Gamma_D^{\mathrm{cusp}}=\bigsqcup_{i=1}^{N}\Gamma^{(i)}
\]
into full subgraphs such that there are no edges between \(\Gamma^{(i)}\) and
\(\Gamma^{(j)}\) for \(i\neq j\), and each restriction
\[
p_{D,D_2}:\Gamma^{(i)}\xrightarrow{\sim}\Gamma_{D_2}^{\mathrm{cusp}}
\]
is an isomorphism of graphs.
\end{theorem}

\begin{proof}
 Let \(\Gamma_{D_2}^{\circ}\) be a connected component of
\(\Gamma_{D_2}^{\mathrm{cusp}}\), and choose
\[
(\mathcal P_G,\phi_{D_2})\in \Gamma_{D_2}^{\circ}.
\]   
By Theorem~\ref{thm:deep_cusp_reductive}(ii), the image of
\[
\Aut(\cP_G,\phi_{D_2})\longrightarrow \Aut(\mathcal \cP_G|_{D_1})
\]
is
\[
T(k)\ltimes U(\mathcal O_{D_1})
\quad\text{if }D_2=0,
\]
and
\[
U(\mathcal O_{D_1})
\quad\text{if }D_2\neq 0.
\]
Consequently, the same counting argument as in Corollary~\ref{cor:deep_cusp_fiber_size_reductive}
gives
\[
\#\,p_{D,D_2}^{-1}(\mathcal P_G,\phi_{D_2})
=
|T(\mathcal O_{D_1})|/|T(k)|^{\delta_{D_2,0}}
\cdot |(G/B)(\mathcal O_{D_1})|=:N.
\]
Choose representatives
\[
p_{D,D_2}^{-1}(\mathcal P_G,\phi_{D_2})
=
\{(\mathcal P_G,\psi_i)\}_{i=1}^N.
\]
For each \(i\), let \(\Gamma^{(i)}\) be the connected component of
\(p_{D,D_2}^{-1}(\Gamma_{D_2}^{\circ})\) containing \((\mathcal P_G,\psi_i)\).
We first prove that the components \(\Gamma^{(i)}\) are pairwise distinct. Suppose that
\(\Gamma^{(i)}=\Gamma^{(j)}\). Then there exists a sequence of adjacent vertices in
\(\Gamma_D^{\mathrm{cusp}}\) joining \((\mathcal P_G,\psi_i)\) to
\((\mathcal P_G,\psi_j)\). Composing the corresponding Hecke modifications gives an
isomorphism
\[
f:\mathcal P_G|_{X\setminus\{x\}}
\xrightarrow{\sim}
\mathcal P_G|_{X\setminus\{x\}}.
\]
By Proposition~\ref{prop:mu-mod-preserves-BHN}, every edge in the cusp preserves the
canonical Harder--Narasimhan \(B\)-reduction. Hence \(f\) preserves the canonical
\(B\)-reduction \(\mathcal P_B\) of \(\mathcal P_G\).

Decompose $D=D'+d_x[x]$ with $x\notin\supp D'$. Choose a \(B\)-compatible trivialization of \(\mathcal P_G\) along \(D'\), and set
\[
g:=\psi_j^{-1}\circ\psi_i\in \Aut(\cP_G|_{D'})\simeq G(\mathcal O_{D'}).
\]
Since \(f\) preserves \(\mathcal P_B\) and $g=f|_{D'}$, the element \(g\) preserves the corresponding
\(B\)-subbundle of the trivial \(G\)-bundle over \(D'\). Therefore
\[
g\in B(\mathcal O_{D'})
=
T(\mathcal O_{D'})\ltimes U(\mathcal O_{D'}).
\]
Write \(g=g_Tg_U\), with
\[
g_T\in T(\mathcal O_{D'}),
\qquad
g_U\in U(\mathcal O_{D'}).
\]
Passing to \(\mathcal P_T:=\mathcal P_B/U\), the map \(f\) induces an automorphism of
\(\mathcal P_T|_{X\setminus\{x\}}\). For every character \(\chi\in X^*(T)\), this gives an
automorphism of
\[
\mathcal L_\chi|_{X\setminus\{x\}},
\qquad
\mathcal L_\chi:=\mathcal P_T\times^{T,\chi}k.
\]
Such an automorphism is multiplication by an element of
\(\Gamma(X\setminus\{x\},\mathcal O_X^\times)=k^\times\). Hence the induced automorphism
of \(\mathcal P_T\) extends uniquely to \(X\), and is given by an element of \(T(k)\). Thus
\[
g_T\in T(k).
\]

By assumption, $(\cP_G,\psi^i|_{D_2})\simeq (\cP_G,\psi^j|_{D_2})$. Fix an automorphism $a$ of $\cP_G$ which sends $\psi^j|_{D_2}$ to $\psi^i|_{D_2}$. As before, we have $a\in T(k)U(\cO_{D_2})$, which allows us to write $a=a_Ta_U$ for $a_T\in T(k)$ and $a_U\in U(\cO_{D_2})$. We claim that the automorphism 
$$
(f|_{D_1},a|_{D_2})\in \Aut(\cP_G|_{D_1})\times\Aut(\cP_G|_{D_2})=\Aut(\cP_G|_D)
$$
lifts to an automorphism of $\cP_G$. Such an automorphism will send $\psi^j$ to $\psi^i$, which will show that \((\mathcal P_G,\psi_i)\simeq(\mathcal P_G,\psi_j)\), so \(i=j\). Equivalently, we wish to prove that
$$
((a^{-1}\circ f)|_{D_1},\mathrm{id}|_{D_2})=(a^{-1}|_{D_1}\circ g,\mathrm{id}|_{D_2})
$$
lifts to an automorphism of $\cP_G$.

If \(D_2=0\), then by
Theorem~\ref{thm:deep_cusp_reductive}(ii), it is enough to prove that 
$$
a^{-1}|_{D_1}\circ g|_{D_1}\in T(k)\ltimes U(\cO_{D_1}),
$$
which is true since both $a|_{D_1}$ and $g$ lie in the group on the right-hand side. 

If \(D_2\neq0\), then by
Theorem~\ref{thm:deep_cusp_reductive}(ii), it is enough to prove that 
$$
a^{-1}|_{D_1}\circ g|_{D_1}\in U(\cO_{D_1}),
$$
which is equivalent to showing that $a_T=g_T$. By assumption, \(D_2\) contains a point \(y\neq x\). Since both $a|_{d_y[y]}$ and $f|_{d_y[y]}$ send $\phi^j|_{d_y[y]}$ to  $\phi^i|_{d_y[y]}$, the composition $(a^{-1}\circ f)|_{d_y[y]}=\mathrm{id}|_{d_y[y]}$. In particular, $(a_T^{-1}g_T)|_{d_y[y]}=\mathrm{id}$. Since $T(k)$ is embedded diagonally into $G(\cO_D)$, this implies that $a_T^{-1}g_T=\mathrm{id}$, as desired.

Thus, the components \(\Gamma^{(i)}\) are pairwise distinct, and hence there are no edges
between \(\Gamma^{(i)}\) and \(\Gamma^{(j)}\) for \(i\neq j\). This shows that the map
\[
p_{D,D_2}:\Gamma^{(i)}\to\Gamma_{D_2}^{\circ}
\]
is a bijection on edges. By Lemma \ref{l:covering_map_G}, it is also a covering map, and hence is an isomorphism of graphs. Taking the union over all connected components
\(\Gamma_{D_2}^{\circ}\) finishes the proof.
\end{proof}

We finish with a description of monodromy of the covering map $p_{D',d_x[x]}$.
\begin{theorem}[Generalization of Theorem \ref{t:graph_for_ramification_at_cusp_over_x}]
\label{thm:deep_cusp_torus_monodromy_reductive}
Assume that
\[
D=D'+d_x[x],
\qquad
x\notin\operatorname{supp}(D'),
\qquad
d_x\geq1.
\]
Take
\[
(\mathcal P_G,\psi^1),\qquad
(\mathcal P_G,\psi^2)
\]
in \(\Gamma_D^{\mathrm{cusp}}\) lying in the same connected component of
\(\Gamma_D^{\mathrm{cusp}}\), and assume that their images in
\(\Gamma_{d_x[x]}^{\mathrm{cusp}}\) are equal. Then there exists an automorphism
\[
\alpha\in \Aut(\mathcal P_G)
\]
which identifies the two \(d_x[x]\)-level structures and satisfies
\[
\alpha\psi^2|_{D'}=t\,\psi^1|_{D'}
\]
for some
\[
t\in T(k).
\]
Here the equality along \(D'\) is written after choosing a \(B\)-compatible trivialization.
\end{theorem}
\begin{proof}
Repeat the argument from the proof of Theorem~\ref{thm:deep_cusp_splitting_reductive}
up to the point where the torus part is shown to lie in \(T(k)\), now with \(D_2=d_x[x]\). It shows that the discrepancy between the two \(D'\)-level structures
lies in
\[
T(k)\ltimes U(\mathcal O_{D'}).
\]
Write it as \(tu\), with \(t\in T(k)\) and \(u\in U(\mathcal O_{D'})\). Since \(d_x\geq1\),
Theorem~\ref{thm:deep_cusp_reductive}(ii) implies that \(u\) is the restriction to \(D'\) of
an automorphism of \(\mathcal P_G\) preserving the \(d_x[x]\)-level structure. Composing by
this automorphism removes the unipotent part and leaves exactly the constant torus factor
\(t\). This gives the required automorphism \(\alpha\).
\end{proof}

\bibliographystyle{alphaurl}
\bibliography{references}
\end{document}